\documentclass[11pt,english]{extarticle}
\usepackage[T1]{fontenc}
\usepackage[latin9]{inputenc}
\usepackage{geometry}
\usepackage{color}
\usepackage{verbatim}
\usepackage{refstyle}
\usepackage{mathtools}
\usepackage{algorithm2e}
\usepackage{amsmath}
\usepackage{amsthm}
\usepackage{amssymb}
\usepackage{graphicx}
\usepackage{esint}

\makeatletter

\AtBeginDocument{\providecommand\eqref[1]{\ref{eq:#1}}}
\AtBeginDocument{\providecommand\lemmaref[1]{\ref{lemma:#1}}}
\providecommand{\tabularnewline}{\\}
\newcommand{\lyxdot}{.}

\RS@ifundefined{subref}
  {\def\RSsubtxt{section~}\newref{sub}{name = \RSsubtxt}}
  {}
\RS@ifundefined{thmref}
  {\def\RSthmtxt{theorem~}\newref{thm}{name = \RSthmtxt}}
  {}
\RS@ifundefined{lemref}
  {\def\RSlemtxt{lemma~}\newref{lem}{name = \RSlemtxt}}
  {}

\theoremstyle{plain}
\newtheorem{thm}{\protect\theoremname}
  \theoremstyle{plain}
  \newtheorem{lem}[thm]{\protect\lemmaname}
  \theoremstyle{plain}
  \newtheorem{cor}[thm]{\protect\corollaryname}
  \theoremstyle{plain}
  \newtheorem*{conjecture*}{\protect\conjecturename}

\usepackage{booktabs}

\date{}

\author{
  Qi Deng \thanks{Department of Computer and Information Science and Engineering, University of Florida, FL, 32611}\\
  \texttt{qdeng@ufl.edu}
  \and
  Guanghui Lan\thanks{Department of Industrial and Systems Engineering, University of Florida, FL, 32611}\\
  \texttt{glan@ise.ufl.edu}
 \and
 Anand Rangarajan\footnotemark[1]\\
 \texttt{anand@cise.ufl.edu}
}

\@ifundefined{showcaptionsetup}{}{%
 \PassOptionsToPackage{caption=false}{subfig}}
\usepackage{subfig}
\makeatother

\usepackage{babel}
  \providecommand{\conjecturename}{Conjecture}
  \providecommand{\corollaryname}{Corollary}
  \providecommand{\lemmaname}{Lemma}
\providecommand{\theoremname}{Theorem}

\begin{document}

\title{Randomized Block Subgradient Methods for Convex Nonsmooth and Stochastic
Optimization}
\maketitle
\begin{abstract}
Block coordinate descent methods and stochastic subgradient methods
have been extensively studied in optimization and machine learning.
By combining randomized block sampling with stochastic subgradient
methods based on dual averaging (\cite{nesterov2009primal,Xiao:2010:DAM}),
we present stochastic block dual averaging (SBDA)---a novel class
of block subgradient methods for convex nonsmooth and stochastic optimization.
SBDA requires only a block of subgradients and updates blocks of variables
and hence has significantly lower iteration cost than traditional
subgradient methods. We show that the SBDA-based methods exhibit the
optimal convergence rate for convex nonsmooth stochastic optimization.
More importantly, we introduce randomized stepsize rules and block
sampling schemes that are adaptive to the block structures, which
significantly improves the convergence rate w.r.t. the problem parameters.
This is in sharp contrast to recent block subgradient methods applied
to nonsmooth deterministic or stochastic optimization (\cite{Dang:arXiv1309.2249,nesterov2014subgradient}).
For strongly convex objectives, we propose a new averaging scheme
to make the regularized dual averaging method optimal, without having
to resort to any accelerated schemes.
\end{abstract}

\section{Introduction\label{sec:Introduction}}

In this paper, we mainly focus on the following convex optimization
problem: 

\begin{equation}
\min_{x\in X}\phi\left(x\right),\label{eq:prob-statement}
\end{equation}
where the feasible set $X$ is embedded in Euclidean space $\mathbb{R}^{N}$
for some integer $N>0$. Letting $N_{1},N_{2},\ldots,N_{n}$ be $n$
positive integers such that $\sum_{i=1}^{n}N_{i}=N$, we assume $X$
can be partitioned as $X=X_{1}\times X_{2}\times\ldots X_{n}$, where
each $X_{i}\subseteq\mathbb{R}^{N_{i}}$. We denote $x\in X$, by
$x=x^{\left(1\right)}\times x^{\left(2\right)}\ldots\times x^{\left(n\right)}$
where $x^{\left(i\right)}\in X_{i}$. The objective $\phi\left(x\right)$
consists of two parts: $\phi\left(x\right)=f\left(x\right)+\omega\left(x\right)$.
We stress that both $f\left(x\right)$ and $\omega\left(x\right)$
can be nonsmooth. $\omega\left(x\right)$ is a  convex function with
block separable structure: $\omega\left(x\right)=\sum_{i=1}^{n}\omega_{i}\left(x_{i}\right)$,
where each $\omega_{i}:X_{i}\rightarrow\mathbb{R}$ is convex and
relatively simple. In composite optimization or regularized learning,
the term $\omega\left(x\right)$ imposes solutions with certain preferred
structures. Common examples of $\omega\left(\cdot\right)$ include
the $\ell_{1}$ norm or squared $\ell_{2}$ norm regularizers. $f\left(x\right)$
is a general convex function. In many important statistical learning
problems,  $f\left(x\right)$  has the form of $f\left(x\right)=\mathbf{E}_{\xi}\left[F\left(x,\xi\right)\right]$,
where $F\left(x,\xi\right)$ is a convex loss function of $x\in X$
with $\xi$ representing sampled data. When it is difficult to evaluate
$f(x)$ exactly, as in batch learning or sample average approximation
(SAA), $f\left(x\right)$ is approximated with finite data. Firstly,
a large number of samples $\xi_{1},\xi_{2},\ldots,\xi_{m}$ are drawn,
and then $f\left(x\right)$ is approximated by $\tilde{f}\left(x\right)=\frac{1}{m}\sum_{i=1}^{m}F\left(x,\xi_{i}\right)$,
with the alternative problem: 

\begin{equation}
\min_{x\in X}\tilde{\phi}(x):=\tilde{f}\left(x\right)+\omega\left(x\right).\label{eq:prob-apprx}
\end{equation}

However, although classic first order methods can provide accurate
solutions to (\ref{eq:prob-apprx}), the major drawback of these approaches
is the poor scalability to large data. First order deterministic methods
require full information of the (sub)gradient and scan through the
entire dataset many times, which is prohibitive for applications where
scalability is paramount. In addition, due to the statistical nature
of the problem, solutions with high precision may not even be necessary.

To solve the aforementioned problems, stochastic methods---stochastic
(sub)gradient descent (SGD) or block coordinate descent (BCD) have
received considerable attention in the machine learning community.
Both of them confer new advantages in the trade offs between speed
and accuracy. Compared to deterministic and full (sub)gradient methods,
they are easier to implement, have much lower computational complexity
in each iteration, and often exhibit sufficiently fast convergence
while obtaining practically good solutions.

SGD was first studied in \cite{robbins1951stochastic} in the 1950s,
with the emphasis mainly on solving strongly convex problems; specifically
it only needs the gradient/subgradient on a few data samples while
iteratively updating all the variables. In the approach of online
learning or stochastic approximation (SA), SGD directly works on the
objective (\ref{eq:prob-statement}), and obtains convergence independent
of the sample size. While early work emphasizes asymptotic properties,
recent work investigate complexity analysis of convergence. Many works
(\cite{Nemirovski:2009:RSA:1654243.1654247,lan2012optimal,singer2009efficient,rakhlin2012making,ghadimi2012optimal,chen2012optimal,hazan2014beyond})
investigate the optimal SGD under various conditions. Proximal versions
of SGD, which explicitly incorporate the regularizer $\omega\left(x\right)$,
have been studied, for example in \cite{langford2009sparse,duchi2009efficient,duchi2010composite,Xiao:2010:DAM}.

The study of BCD also has a long history. BCD was initiated in \cite{luo1991convergence,Luo1992},
but the application of BCD to linear systems dates back to even earlier
(for example see the Gauss-Seidel method in \cite{hageman2012applied}).
It works on the approximated problem (\ref{eq:prob-apprx}) and makes
progress by reducing the original problem into subproblems using only
a single block coordinate of the variable at a time. Recent works
\cite{nesterov2012efficiency,richtarik2014iteration,Shalev-Shwartz:2009:SML:1553374.1553493,lu2013complexity}
study BCD with random sampling (RBCD) and obtain non-asymptotic complexity
rates.  For the regularized learning problem as in (\ref{eq:prob-apprx}),
RBCD on the dual formulation  has been proposed \cite{shalev2012stochastic,lacoste2012block,shalev2013accelerated}.
Although most of the work on BCD focuses on smooth (composite) objectives,
some recent work (\cite{Dang:arXiv1309.2249,xu2014block,wang2014randomized,zhao2014accelerated})
seeks to extend the realm of BCD in various ways. The works in \cite{nesterov2014subgradient,Dang:arXiv1309.2249}
discuss (block) subgradient methods for nonsmooth optimization. Combining
the ideas of SGD and BCD, the works in \cite{Dang:arXiv1309.2249,xu2014block,wang2014randomized,zhao2014accelerated,reddi2014large}
employ sampling of both features and data instances in BCD. 

In this paper, we propose a new class of block subgradient methods,
namely, stochastic block dual averaging (SBDA), for solving nonsmooth
deterministic and stochastic optimization problems. Specifically,
SBDA consists of a new dual averaging step incorporating the average
of all past (stochastic) block subgradients and variable updates involving
only block components. We bring together two strands of research,
namely, the dual averaging algorithm (DA) \cite{Xiao:2010:DAM,nesterov2009primal}
which was studied for nonsmooth optimization and randomized coordinate
descent (RCD) \cite{nesterov2012efficiency}, employed for smooth
deterministic problems. Our main contributions consist of the following:
\begin{itemize}
\item Two types of SBDA have been proposed for different purposes. For regularized
learning, we propose SBDA-u which performs uniform random sampling
of blocks. For more general nonsmooth learning problems, we propose
SBDA-r which applies an optimal sampling scheme with improved convergence.
Compared with existing subgradient methods for nonsmooth and stochastic
optimization, both SBDA-u and SBDA-r have significantly lower iteration
cost when the computation of block subgradients and block updates
are convenient. 
\item We contribute a novel scheme of randomized stepsizes and optimized
sampling strategies which are truly adaptive to the block structures.
Selecting block-wise stepsizes and optimal block sampling have been
critical issues for speeding up BCD for smooth regularized problems,
please see \cite{nesterov2012efficiency,qu2014coordinate,shalev2012stochastic,richtarik2014iteration}
for some recent advances. For nonsmooth or stochastic optimization,
the most closely related work to ours are \cite{Dang:arXiv1309.2249,nesterov2014subgradient}
which do not apply block-wise stepsizes. To the best of our knowledge,
this is the \emph{first time} block subgradient methods with block
adaptive stepsizes and optimized sampling have been proposed for nonsmooth
and stochastic optimization. 
\item We provide new theoretical guarantees of convergence of SBDA methods.
SBDA obtains the optimal rate of convergence for general convex problems,
matching the state of the art results in the literature of stochastic
approximation and online learning. More importantly, SBDA exhibits
a significantly improved convergence rate w.r.t. the problem parameters.
When the regularizer $\omega\left(x\right)$ is strongly convex, our
analysis provides a simple way to make the regularized dual averaging
methods in \cite{Xiao:2010:DAM} optimal. We show an \textit{aggressive}
weighting is sufficient to obtain{\small{} $O\left(\frac{1}{T}\right)$}convergence
where $T$ is the iteration count, without the need for any accelerated
schemes. This appears to be a new result for simple dual averaging
methods.
\end{itemize}

\paragraph*{Related work}

Extending BCD to the realm of nonsmooth and stochastic optimization
has been of interest lately. Efficient subgradient methods for a class
of nonsmooth problems has been proposed in \cite{nesterov2014subgradient}.
However, to compute the stepsize, the block version of this subgradient
method requires computation of the entire subgradient and knowledge
of the optimal value;  hence, it may be not efficient in a more general
setting. The methods in \cite{Dang:arXiv1309.2249,nesterov2014subgradient}
employ stepsizes that are not adaptive to the block selection and
have  therefore suboptimal bounds to our work. For SA or online learning,
SBDA applies double sampling of both blocks and data. A similar approach
has also been employed for new stochastic methods in some very recent
work (\cite{Dang:arXiv1309.2249,zhao2014accelerated,wang2014randomized,reddi2014large,xu2014block}).
It should be noted here that if the assumptions are strengthened,
namely, in the batch learning formulation, and if $\tilde{\phi}$
is smooth, it is possible to obtain a linear convergence rate $O\left(e^{-T}\right)$.
Nesterov's randomized block coordinate methods \cite{nesterov2012efficiency,richtarik2014iteration}
consider different stepsize rules and block sampling but only for
smooth objectives with possible nonsmooth regularizers. Recently,
nonuniform sampling in BCD has been addressed in \cite{qu2014coordinate,zhao2014stochastic,yintat2013coordinatedescent}
and shown to have advantages over uniform sampling. Although our work
discusses block-wise stepsizes and nonuniform sampling as well, we
stress the nonsmooth objectives that appear in deterministic and stochastic
optimization  . The proposed algorithms employ very different proof
techniques, thereby obtaining different optimized sampling distributions.

\subsubsection*{Outline of the results. }

We introduce two versions of SBDA that are appropriate in different
contexts. The first algorithm, SBDA with uniform block sampling (SBDA-u)
works for a class of convex composite functions, namely, $\omega\left(x\right)$
is explicate in the proximal step. When $\omega(x)$ is a general
convex function, for example, the sparsity regularizer $\|x\|_{1}$,
we show that SBDA-u obtains the convergence rate of $O\left(\frac{\sqrt{n}\sum_{i}^{n}\sqrt{M_{i}^{2}D_{i}}}{\sqrt{T}}\right)$,
which improves the rate of $O\left(\frac{\sqrt{n}\sqrt{\sum_{i}^{n}M_{i}^{2}}\cdot\sqrt{\sum_{i}^{n}D_{i}}}{\sqrt{T}}\right)$
by SBMD. Here $\{M_{i}\}$ and $\{D_{i}\}$ are some parameters associated
with the blocks of coordinates to be specified later. When $\omega(x)$
is a strongly convex function, by using a more aggressive scheme to
be later specified, SBDA-u obtains the optimal rate of $O\left(\frac{n\sum_{i}M_{i}^{2}}{\lambda T}\right)$,
matching the result from SBMD. In addition, for general convex problems
in which $\omega(x)=0$, we propose a variant of SBDA with nonuniform
random sampling (SBDA-r) which achieves an improved convergence rate
{\small{}{}$O\left(\frac{\left(\sum_{j=1}^{n}M_{j}^{2/3}D_{j}^{1/3}\right)^{3/2}}{\sqrt{T}}\right)$}.
These computational results are summarized in Table~(\ref{tab:complexity}). 

\begin{table}
\centering{}%
\begin{tabular}{|c|c|c|}
\hline 
Algorithm  & Objective  & Complexity\tabularnewline
\hline 
\hline 
SBDA-u  & Convex composite  & $O\left(\frac{\sqrt{n}\sum_{i}^{n}\sqrt{M_{i}^{2}D_{i}}}{\sqrt{T}}\right)$\tabularnewline
\hline 
SBDA-u  & Strongly convex composite  & $O\left(\frac{n\sum_{i}M_{i}^{2}}{\lambda T}\right)$\tabularnewline
\hline 
SBDA-r  & Convex nonsmooth  & $O\left(\frac{\left(\sum_{j=1}^{n}M_{j}^{2/3}D_{j}^{1/3}\right)^{3/2}}{\sqrt{T}}\right)$\tabularnewline
\hline 
\end{tabular}\protect\caption{\label{tab:complexity}Iteration complexity of our SBDA algorithms.}
\end{table}

\paragraph*{Structure of the Paper}

The paper proceeds as follows. Section 2 introduces the notation used
in this paper. Section 3 presents and analyzes SBDA-u. Section 4 presents
SBDA-r, and discusses optimal sampling and its convergence. Experimental
results to demonstrate the performance of SBDA are provided in section
6. Section 7  draws conclusion and comments on possible future directions.

\section{Preliminaries}

Let $\mathbb{R}^{N}$ be a Euclidean vector space, $N_{1},N_{2},\ldots N_{n}$
be $n$ positive integers such that $N_{1}+\ldots N_{n}=N$. Let $I$
be the identity matrix in $\mathbb{R}^{N\times N}$, $U_{i}$ be a
$N\times N_{i}$-dim matrix such that 
\[
I=\left[U_{1}U_{2}\ldots U_{n}\right].
\]
For each $x\in\mathbb{R}^{N}$, we have the decomposition: $x=U_{1}x^{\left(1\right)}+U_{2}x^{\left(2\right)}+\ldots+U_{n}x^{\left(n\right)}$,
where $x^{\left(i\right)}\in\mathbb{R}^{N_{i}}$.

Let $\|\cdot\|_{\left(i\right)}$ denote the norm on the $\mathbb{R}^{N_{i}}$,
and $\|\cdot\|_{\left(i\right),*}$ be the  induced dual norm. We
define the norm $\|\cdot\|$ in $\mathbb{R}^{N}$ by: $\|x\|^{2}=\sum_{i=1}^{n}\|x^{\left(i\right)}\|_{\left(i\right)}^{2}$
and its dual norm: $\|\cdot\|_{*}$ by $\|x\|_{*}^{2}=\sum_{i=1}^{n}\|x^{\left(i\right)}\|_{\left(i\right),*}^{2}$

Let $d_{i}:X_{i}\rightarrow\mathbb{R}$ be a distance transform function
 with modulus $\|\cdot\|_{\left(i\right)}$ with respect to $\rho$..
  $d_{i}\left(\cdot\right)$ is continuously differentiable and strongly
convex: 
\[
d_{i}\left(\alpha x+\left(1-\alpha\right)y\right)\le\alpha d_{i}\left(x\right)+\left(1-\alpha\right)d_{i}\left(y\right)-\frac{1}{2}\rho\alpha\left(1-\alpha\right)\|x-y\|_{\left(i\right)}^{2},\quad x,y\in X_{i},
\]
$i=1,2,...,n$. 

Let us assume there exists a solution $x^{*}\in X$ to the problem
(\ref{eq:prob-statement})  , and

\begin{equation}
d_{i}(x^{*(i)})\le D_{i}<\infty,\quad i=1,2,\ldots n,\label{eq:assum-D}
\end{equation}
Without loss of generality, we assume $d_{i}\left(\cdot\right)$ is
nonnegative, and write 
\begin{equation}
d\left(x\right)=\sum_{i}^{n}d_{i}\left(x^{\left(i\right)}\right)\label{eq:defi-prox-func}
\end{equation}
 for simplicity. Further more, we define the Bregman divergence associated
with $d_{i}\left(\cdot\right)$ by
\[
\mathcal{V}_{i}\left(z,x\right)=d_{i}\left(x\right)-d_{i}\left(z\right)-\left\langle \nabla_{i}d\left(z\right),x-z\right\rangle ,\quad z,x\in X_{i}.
\]
and $\mathcal{V}\left(z,x\right)=\sum_{i}^{n}\mathcal{V}_{i}\left(z^{\left(i\right)},x^{\left(i\right)}\right)$.

We denote $f(x)=E_{\xi}\left[F\left(x,\xi\right)\right]$, and let
$G(x,\xi)$ be a subgradient of $F(x,\xi)$, and $g(x)=E_{\xi}\left[G\left(x,\xi\right)\right]\in\partial f\left(x\right)$
be a subgradient of $f(x)$. Let $g^{\left(i\right)}\left(\cdot\right)$,
$G^{\left(i\right)}\left(x,\xi\right)$ denote their $i$-th block
components , for $i=1,2,\ldots,n$. Throughout the paper, we assume
the (stochastic) block coordinate subgradient of $f$ satisfying:
\begin{equation}
\|g^{(i)}\left(x\right)\|_{\left(i\right),*}^{2}=\mathbf{E}^{2}\left[\|G^{\left(i\right)}\left(x,\xi\right)\|_{\left(i\right),*}\right]\le\mathbf{E}\left[\|G^{\left(i\right)}\left(x,\xi\right)\|_{\left(i\right),*,}^{2}\right]\le M_{i}^{2},\quad\forall x\in X\label{ass-grad-bound}
\end{equation}
for $i=1,2,\ldots,n$. Note that  although we make assumptions of
stochastic objective , the following analysis and conclusions naturally
extend to deterministic optimization. To see that, we can simply assume
$g\left(x\right)\equiv G\left(x,\xi\right)$, and $f\left(x\right)\equiv F\left(x,\xi\right)$,
for any $\xi$.

 Before introducing the main convergence properties, we first summarize
several useful results in the following lemmas.  Lemma \ref{lem:prox-ineq-1},
\ref{lem:prox-ineq-2}, and \ref{lem:function-bound} slightly generalize
the results in \cite{tseng2008accelerated,lan2011primal}, \cite{nesterov2009primal},
and \cite{lan2012optimal} respectively; their proofs are left in
Appendix.
\begin{lem}
\label{lem:prox-ineq-1}Let $f\left(\cdot\right)$ be a lower semicontinuous
convex function and $d\left(\cdot\right)$ be defined by (\ref{eq:defi-prox-func}).
If 
\[
z=\arg\min_{x}\Psi\left(x\right)\coloneqq f\left(x\right)+d\left(x\right),
\]
 then 
\[
\Psi\left(x\right)\ge\Psi\left(z\right)+\mathcal{V}\left(z,x\right),\quad\forall x\in X.
\]
Moreover, if $f\left(x\right)$ is $\lambda$-strongly convex with
norm $\|\cdot\|_{\left(i\right)}$, and  $x=z+U_{i}y\in X$ where
$y\in X_{i}$, $z\in X$, then 
\[
\Psi\left(x\right)\ge\Psi\left(z\right)+\mathcal{V}\left(z,x\right)+\frac{\lambda}{2}\|y\|_{\left(i\right)}^{2},\quad\forall x\in X.
\]

\end{lem}

\begin{lem}
\label{lem:prox-ineq-2} Let $\Psi:X\rightarrow\mathbb{R}$ be convex,
block  separable, and  $\rho_{i}$-strongly convex with modulus $\rho_{i}$
w.r.t. $\|\cdot\|_{\left(i\right)}$ , $\rho_{i}>0$, $1\le i\le n$,
and $g\in\mathbb{R}^{N}$. If 
\[
x_{0}\in\arg\min_{x\in X}\left\{ \Psi\left(x\right)\right\} ,\text{and}\;z\in\arg\min_{x\in X}\left\{ \left\langle U_{i}g^{\left(i\right)},x\right\rangle +\Psi\left(x\right)\right\} ,
\]
then

\[
\left\langle U_{i}g^{\left(i\right)},x_{0}\right\rangle +\Psi\left(x_{0}\right)\le\left\langle U_{i}g^{\left(i\right)},z\right\rangle +\Psi\left(z\right)+\frac{1}{2\rho_{i}}\|g\|_{\left(i\right),*}^{2}.
\]

\end{lem}

\begin{lem}
\label{lem:function-bound}If $f$ satisfies the assumption (\ref{ass-grad-bound}),
let $x=z+U_{i}y\in X$ where $y\in X_{i}$, $x\in X$, then 
\begin{equation}
f(z)\le f(x)+\langle g^{(i)}(x),y\rangle+2M_{i}\|y\|_{(i)}.
\end{equation}

\end{lem}

\section{Uniformly randomized SBDA (SBDA-u)\label{sec:Uniformly-randomized-SBDA}}

In this section, we describe uniformly randomized SBDA (SBDA-u) for
the composite problem (\ref{eq:prob-statement}). We consider the
 formulation proposed in \cite{Xiao:2010:DAM}, since it incorporates
the regularizers for composite problems. The main update of the DA
algorithm has the form 
\begin{equation}
x_{t+1}=\arg\min_{x\in X}\left\{ \sum_{s=1}^{t}\left\langle G_{s},x\right\rangle +t\omega\left(x\right)+\beta_{t}d\left(x\right)\right\} ,\label{eq:dual-averaging0}
\end{equation}
where $\left\{ \beta_{t}\right\} $ is a parameter sequence and $G_{s}$
is shorthand for $G\left(x_{s},\xi_{s}\right)$, and $d\left(x\right)$
is a strongly convex proximal function. When $\omega\left(x\right)=0$,
this reduces to a version of Nesterov's primal-dual subgradient method
\cite{nesterov2009primal}. 

Let $\bar{G}=\sum_{s=0}^{t}\alpha_{s}U_{i_{s}}G^{\left(i_{s}\right)}\left(x_{s},\xi_{s}\right)$,
where $\left\{ \alpha_{t}\right\} $ is a sequence of positive values,
$\left\{ i_{t}\right\} $ is a sequence of sampled indices. The main
iteration step of SBDA has the form 
\begin{equation}
x_{t+1}^{\left(i_{t}\right)}=\arg\min_{x\in X_{i_{t}}}\left\{ \langle\bar{G}^{\left(i_{t}\right)},x\rangle+l_{t}^{\left(i_{t}\right)}\omega_{i_{t}}\left(x\right)+\gamma_{t}^{\left(i_{t}\right)}d_{i_{t}}\left(x\right)\right\} ,\label{eq:dual-averaging}
\end{equation}
and $x_{t+1}^{\left(i\right)}=x_{t}^{\left(i\right)}$, $i\neq i_{t}$. 

We highlight two important aspects of the proposed iteration (\ref{eq:dual-averaging}).
Firstly, the update in (\ref{eq:dual-averaging}) incorporates the
past randomly sampled block (stochastic) subgradients $\left\{ G^{\left(i_{t}\right)}\left(x_{t},\xi_{t}\right)\right\} $,
rather than the full (stochastic) subgradients. Meanwhile, the update
of the primal variable is restricted to the same block ($i_{t}$),
leaving the other blocks untouched. Such block decomposition significantly
reduces the iteration cost of the dual averaging method when the block-wise
operation is convenient. Secondly, (\ref{eq:dual-averaging}) employs
a novel randomized stepsize sequence $\left\{ \gamma_{t}\right\} $
where $\gamma_{t}\in\mathbb{R}^{n}$. More specifically, $\gamma_{t}$
depends not only on the iteration count $t$, but also on the block
index $i_{t}$. $\left\{ \gamma_{t}\right\} $ satisfies the assumptions,
\begin{align}
\gamma_{t}^{\left(j\right)}=\gamma_{t-1}^{\left(j\right)},j\neq i_{t},\,\mathrm{and}\,\gamma_{t}^{\left(j\right)} & \ge\gamma_{t-1}^{\left(j\right)},j=i_{t}.\label{eq:assumption-gamma}
\end{align}
The most important aspect of (\ref{eq:assumption-gamma}) is that
stepsizes can be specified for each block of coordinates, thereby
allowing for aggressive descent. As will be shown later, the rate
of convergence, in terms of the problem parameters, can be significantly
reduced by properly choosing these control parameters. In addition,
we allow the sequence $\left\{ \alpha_{t}\right\} $ and the related
$\left\{ l_{t}\right\} $ to be variable, hence offer the opportunity
of different averaging schemes in composite settings. To summarize,
the full SBDA-u is described in Algorithm~\ref{alg:sbda-u}.

\begin{algorithm}
\KwIn{convex composite function $\phi\left(x\right)=f\left(x\right)+\omega\left(x\right)$,
a sequence of samples $\left\{ \xi_{t}\right\} $;} 

initialize $\alpha_{-1}\in\mathbb{R}$, $\gamma_{-1}\in\mathbb{R}^{n}$,
$l_{-1}=\mathbf{0}^{n}$, $\bar{G}=\mathbf{0}^{N}$, $x_{0}=\arg\min_{x\in X}\sum_{i=1}^{n}\gamma_{-1}^{\left(i\right)}d_{i}\left(x^{\left(i\right)}\right)$\;

\For{ $t=0,1,\ldots,T-1$}{ sample a block $i_{t}\in\left\{ 1,2,\ldots,n\right\} $
with uniform probability $\frac{1}{n}$\;

set $\gamma_{t}^{\left(i\right)}$, $i=1,2,\ldots,n$\; 

set $l_{t}^{\left(i_{t}\right)}=l_{t-1}^{\left(i_{t}\right)}+\alpha_{t}$
and $l_{t}^{\left(j\right)}=l_{t-1}^{\left(j\right)}$ for $j\neq i_{t}$\;

update $\bar{G}$: $\bar{G}=\bar{G}+\alpha_{t}U_{i_{t}}G^{\left(i_{t}\right)}\left(x_{t},\xi_{t}\right)$\; 

update $x_{t+1}^{\left(i_{t}\right)}=\arg\min_{x\in X_{i_{t}}}\left\{ \langle\bar{G}^{\left(i_{t}\right)},x\rangle+l_{t}^{\left(i_{t}\right)}\omega_{i_{t}}\left(x\right)+\gamma_{t}^{\left(i_{t}\right)}d_{i_{t}}\left(x\right)\right\} $\; 

$x_{t+1}^{\left(j\right)}=x_{t}^{\left(j\right)}$, for $j\neq i_{t}$\;
}

\KwOut{$\bar{x}=\left[\sum_{t=1}^{T}\left(\alpha_{t-1}-\frac{n-1}{n}\alpha_{t}\right)x_{t}\right]/\sum_{t=1}^{T}\left(\alpha_{t-1}-\frac{n-1}{n}\alpha_{t}\right)$;}\caption{Uniformly randomized stochastic block dual averaging (SBDA-u) method.
\label{alg:sbda-u}}
\end{algorithm}

The following theorem illustrates an important relation to analyze
the convergence of SBDA-u. Throughout the analysis we assume the simple
function $\omega(x)$ is $\lambda$-strongly convex with modulus $\lambda$,
where $\lambda\ge0$. 
\begin{thm}
\label{thm-composite-conv}In algorithm \ref{alg:sbda-u}, if the
sequence $\left\{ \gamma_{t}\right\} $ satisfies the assumption (\ref{eq:assumption-gamma})
, then for any $x\in X$, we have 
\begin{eqnarray}
\sum_{t=1}^{T}\left(\alpha_{t-1}-\frac{n-1}{n}\alpha_{t}\right)\mathbf{E}\left[\phi\left(x_{t}\right)-\phi\left(x\right)\right] & \le & \alpha_{0}\frac{n-1}{n}\left[\phi\left(x_{0}\right)-\phi\left(x\right)\right]+\sum_{i=1}^{n}\mathbf{E}\left[\gamma_{T-1}^{\left(i\right)}\right]d_{i}\left(x\right)\nonumber \\
 &  & +\sum_{i=1}^{n}\frac{5M_{i}^{2}}{n}\sum_{t=0}^{T-1}\mathbf{E}\left[\frac{\alpha_{t}^{2}}{\left(\gamma_{t-1}^{\left(i\right)}\rho+l_{t-1}^{\left(i\right)}\lambda\right)}\right].\label{main-converg-composite}
\end{eqnarray}
\end{thm}
\begin{proof}
Firstly, to simplify the notation, when there is no ambiguity, we
use the terms $\omega_{i}\left(x\right)$ and $\omega_{i}\left(x^{\left(i\right)}\right)$,
$d_{i}\left(x\right)$ and $d_{i}\left(x^{\left(i\right)}\right)$,
$\mathcal{V}_{i}\left(x,y\right)$ and $\mathcal{V}_{i}\left(x^{\left(i\right)},y^{\left(i\right)}\right)$
interchangeably. In addition, we denote $\omega_{i^{c}}\left(x\right)=\omega\left(x\right)-\omega_{i}\left(x\right)$,
and an auxiliary function by 
\begin{equation}
\Psi_{t}\left(x\right)=\begin{cases}
\sum_{s=0}^{t}\alpha_{s}\left[F\left(x_{s},\xi_{s}\right)+\left\langle G_{s},x-x_{s}\right\rangle _{\left(i_{s}\right)}+\omega_{i_{s}}\left(x\right)\right]+\sum_{i=1}^{n}\gamma_{t}^{\left(i\right)}d_{i}\left(x^{\left(i\right)}\right), & t\ge0\\
\sum_{i=1}^{n}\gamma_{t}^{\left(i\right)}d_{i}\left(x^{\left(i\right)}\right) & t=-1
\end{cases}.\label{eq:defi-Psi-func}
\end{equation}
It can be easily seen from the definition that $x_{t+1}$ is the minimizer
of the problem $\min_{x\in X}\Psi_{t}\left(x\right)$.  Moreover,
by the assumption on $\left\{ \gamma_{t}\right\} $, we obtain

\begin{equation}
\Psi_{t}\left(x\right)-\Psi_{t-1}\left(x\right)\ge\alpha_{t}\left[F\left(x_{t},\xi_{t}\right)+\left\langle G_{t},x-x_{t}\right\rangle _{\left(i_{t}\right)}+\omega_{i_{t}}\left(x\right)\right].\quad t=0,1,2,\ldots\label{eq:recursive-Psi}
\end{equation}
Applying Lemma \ref{lem:function-bound} and the property \eqref{recursive-Psi}
at $x=x_{t+1}$, we have 
\begin{eqnarray*}
\phi\left(x_{t+1}\right) & \le & f\left(x_{t}\right)+\left\langle g_{t},x_{t+1}-x_{t}\right\rangle +2M_{i_{t}}\|x_{t+1}-x_{t}\|_{\left(i_{t}\right)}+\omega\left(x_{t+1}\right)\\
 & = & F\left(x_{t},\xi_{t}\right)+\left\langle G_{t},x_{t+1}-x_{t}\right\rangle _{\left(i_{t}\right)}+2M_{i_{t}}\|x_{t+1}-x_{t}\|_{\left(i_{t}\right)}\\
 &  & +f\left(x_{t}\right)-F\left(x_{t},\xi_{t}\right)+\left\langle g_{t}-G_{t},x_{t+1}-x_{t}\right\rangle _{\left(i_{t}\right)}+\omega\left(x_{t+1}\right)\\
 & \le & \frac{1}{\alpha_{t}}\underset{\Delta_{1}}{\underbrace{\left[\Psi_{t}\left(x_{t+1}\right)-\Psi_{t-1}\left(x_{t+1}\right)+\frac{\gamma_{t-1}^{\left(i_{t}\right)}\rho+l_{t-1}^{\left(i_{t}\right)}\lambda}{2}\|x_{t+1}-x_{t}\|_{\left(i_{t}\right)}^{2}\right]}}\\
 &  & +f\left(x_{t}\right)-F\left(x_{t},\xi_{t}\right)+\omega_{i_{t}^{c}}\left(x_{t+1}\right)\\
 &  & +\underset{\Delta_{2}}{\underbrace{\left\langle g_{t}-G_{t},x_{t+1}-x_{t}\right\rangle _{\left(i_{t}\right)}-\frac{\gamma_{t-1}^{\left(i_{t}\right)}\rho+l_{t-1}^{\left(i_{t}\right)}\lambda}{2\alpha_{t}}\|x_{t+1}-x_{t}\|_{\left(i_{t}\right)}^{2}+2M_{i_{t}}\|x_{t+1}-x_{t}\|_{\left(i_{t}\right)}}}.
\end{eqnarray*}

We proceed with the analysis by separately taking care of $\Delta_{1}$
and $\Delta_{2}$.  We first provide a concrete bound on $\Delta_{1}$.
Applying Lemma~\ref{lem:prox-ineq-1} for  $\Psi=\Psi_{t-1}$ with
$x_{t}$ being the optimal point $x=x_{t+1}$, we obtain 
\begin{equation}
\Psi_{t-1}\left(x_{t+1}\right)\ge\Psi_{t-1}\left(x_{t}\right)+\sum_{i=1}^{n}\gamma_{t-1}^{\left(i\right)}\mathcal{V}_{i}\left(x_{t},x_{t+1}\right)+\frac{l_{t-1}^{\left(i_{t}\right)}\lambda}{2}\|x_{t}-x_{t+1}\|_{\left(i_{t}\right)}^{2}.\label{eq:bound3}
\end{equation}
In view of (\ref{eq:bound3}) and the assumption $\mathcal{V}_{i}\left(x_{t},x_{t+1}\right)\ge\frac{\rho}{2}\|x_{t+1}-x_{t}\|_{\left(i\right)}^{2}$,
we obtain an upper bound on $\Delta_{1}$: $\Delta_{1}\le\Psi_{t}\left(x_{t+1}\right)-\Psi_{t-1}\left(x_{t}\right).$
On the other hand, from the Cauchy-Schwarz inequality, we have $\left\langle g_{t}-G_{t},x_{t+1}-x_{t}\right\rangle _{\left(i_{t}\right)}\le\|g_{t}-G_{t}\|_{\left(i_{t}\right),*}\cdot\|x_{t+1}-x_{t}\|_{\left(i_{t}\right)}$.
Then 
\[
\Delta_{2}\le\|x_{t+1}-x_{t}\|_{\left(i_{t}\right)}\cdot\left(\|g_{t}-G_{t}\|_{\left(i_{t}\right)}+2M_{i_{t}}\right)-\frac{\gamma_{t-1}^{\left(i_{t}\right)}\rho+l_{t-1}^{\left(i_{t}\right)}\lambda}{2\alpha_{t}}\|x_{t+1}-x_{t}\|_{\left(i_{t}\right)}^{2}.
\]
The right side of the above inequality is a quadratic function of
$\|x_{t+1}-x_{t}\|_{\left(i_{t}\right)}$. By maximizing it, we obtain
\[
\Delta_{2}\le\frac{\alpha_{t}\left(\|g_{t}-G_{t}\|_{\left(i_{t}\right)}+2M_{i_{t}}\right)^{2}}{2\left(\gamma_{t-1}^{\left(i_{t}\right)}\rho+l_{t-1}^{\left(i_{t}\right)}\lambda\right)}.
\]
In view of these bounds on $\Delta_{1}$ and $\Delta_{2}$, and the
fact that $\omega_{i_{t}^{c}}\left(x_{t}\right)=\omega_{i_{t}^{c}}\left(x_{t+1}\right)$,
we have 
\begin{eqnarray}
\alpha_{t}\phi\left(x_{t+1}\right) & \le & \Psi_{t}\left(x_{t+1}\right)-\Psi_{t-1}\left(x_{t}\right)+\alpha_{t}\left[f\left(x_{t}\right)-F\left(x_{t},\xi_{t}\right)+\omega_{i_{t}^{c}}\left(x_{t}\right)\right]\nonumber \\
 &  & +\frac{\alpha_{t}^{2}\left(\|g_{t}-G_{t}\|_{\left(i_{t}\right)}+2M_{i}\right)^{2}}{\gamma_{t-1}^{\left(i_{t}\right)}\rho+l_{t-1}^{\left(i_{t}\right)}\lambda}.\label{eq:bound-06}
\end{eqnarray}
Summing up the above for $t=0,1,\ldots,T-1$, and observing that $\Psi_{-1}\ge0$,
$d_{i}\left(x_{0}\right)\ge0$ $\left(1\le i\le n\right)$, we obtain
\begin{eqnarray}
\sum_{t=0}^{T-1}\alpha_{t}\phi\left(x_{t+1}\right) & \le & \Psi_{T-1}\left(x_{T}\right)+\sum_{t=0}^{T-1}\frac{\alpha_{t}^{2}\left(\|g_{t}-G_{t}\|_{\left(i_{t}\right)}+2M_{i}\right)^{2}}{\gamma_{t-1}^{\left(i_{t}\right)}\rho+l_{t-1}^{\left(i_{t}\right)}\lambda}\nonumber \\
 &  & +\sum_{t=0}^{T-1}\alpha_{t}\left[f\left(x_{t}\right)-F\left(x_{t},\xi_{t}\right)+\omega_{i_{t}^{c}}\left(x_{t}\right)\right].\label{eq:bound-02}
\end{eqnarray}
Due to the optimality of $x_{T}$, for $x\in X$, we have 
\begin{eqnarray}
\Psi_{T-1}\left(x_{T}\right) & \le & \Psi_{T-1}\left(x\right)\nonumber \\
 & = & \sum_{t=0}^{T-1}\alpha_{t}\left[f\left(x_{t}\right)+\frac{1}{n}\left\langle g_{t},x-x_{t}\right\rangle +\omega_{i_{t}}\left(x^{*}\right)\right]+\sum_{i=1}^{n}\gamma_{T-1}^{\left(i\right)}d_{i}\left(x\right)\nonumber \\
 &  & +\sum_{t=0}^{T-1}\alpha_{t}\left[\left\langle G_{t},x-x_{t}\right\rangle _{\left(i_{t}\right)}-\frac{1}{n}\left\langle g_{t},x-x_{t}\right\rangle \right]\nonumber \\
 & \le & \sum_{t=0}^{T-1}\alpha_{t}\left[\frac{n-1}{n}f\left(x_{t}\right)+\frac{1}{n}f\left(x\right)+\omega_{i_{t}}\left(x\right)\right]+\sum_{i=1}^{n}\gamma_{T-1}^{\left(i\right)}d_{i}\left(x\right)\nonumber \\
 &  & +\sum_{t=0}^{T-1}\alpha_{t}\left[\left\langle G_{t},x-x_{t}\right\rangle _{\left(i_{t}\right)}-\frac{1}{n}\left\langle g_{t},x-x_{t}\right\rangle \right],\label{eq:Phi-bound}
\end{eqnarray}
where the last inequality follows from the convexity of $f$:$\left\langle g_{t},x-x_{t}\right\rangle \le f\left(x\right)-f\left(x_{t}\right)$.
Putting (\ref{eq:bound-02}) and (\ref{eq:Phi-bound}) together yields
\begin{eqnarray}
\sum_{t=0}^{T-1}\alpha_{t}\phi\left(x_{t+1}\right) & \le & \sum_{t=0}^{T-1}\alpha_{t}\left[\frac{n-1}{n}\phi\left(x_{t}\right)+\frac{1}{n}\phi\left(x\right)\right]+\sum_{i=1}^{n}\gamma_{T-1}^{\left(i\right)}d_{i}\left(x\right)\nonumber \\
 &  & +\sum_{0}^{T-1}\frac{\alpha_{t}^{2}\left(\|g_{t}-G_{t}\|_{\left(i_{t}\right)}+2M_{i}\right)^{2}}{2\left(\gamma_{t-1}^{\left(i_{t}\right)}\rho+l_{t-1}^{\left(i_{t}\right)}\lambda\right)}+\delta_{T},\label{eq:aux4}
\end{eqnarray}
where  $\delta_{T}$ is defined by

\begin{eqnarray}
\delta_{T} & = & \sum_{t=0}^{T-1}\alpha_{t}\left[\left\langle G_{t},x-x_{t}\right\rangle _{\left(i_{t}\right)}-\frac{1}{n}\left\langle g_{t},x-x_{t}\right\rangle +f\left(x_{t}\right)-F\left(x_{t},\xi_{t}\right)\right]\nonumber \\
 &  & +\sum_{t=0}^{T-1}\alpha_{t}\left[\omega_{i_{t}^{c}}\left(x_{t}\right)-\frac{n-1}{n}\omega\left(x_{t}\right)+\omega_{i_{t}}\left(x\right)-\frac{1}{n}\omega\left(x\right)\right].\label{eq:defi-delta}
\end{eqnarray}
In (\ref{eq:aux4}), subtracting $\sum_{t=0}^{T-1}\phi\left(x\right)$,
and then $\frac{n-1}{n}\sum_{t=1}^{T}\alpha_{t}\left[\phi\left(x_{t}\right)-\phi\left(x\right)\right]$
on both sides , one has 
\begin{eqnarray}
\sum_{t=1}^{T}\left(\alpha_{t-1}-\frac{n-1}{n}\alpha_{t}\right)\left[\phi\left(x_{t}\right)-\phi\left(x\right)\right] & \le & \frac{n-1}{n}\alpha_{0}\left[\phi\left(x_{0}\right)-\phi\left(x\right)\right]+\sum_{i=1}^{n}\gamma_{T-1}^{\left(i\right)}d_{i}\left(x\right)\nonumber \\
 &  & +\delta_{T}+\sum_{0}^{T-1}\frac{\alpha_{t}^{2}\left(\|g_{t}-G_{t}\|_{\left(i_{t}\right)}+2M_{i_{t}}\right)^{2}}{2\left(\gamma_{t-1}^{\left(i_{t}\right)}\rho+l_{t-1}^{\left(i_{t}\right)}\lambda\right)}.\label{eq:bound-05}
\end{eqnarray}

Now let us take the expectation on  both sides of (\ref{eq:bound-05}).
Firstly, taking the expectation with respect to $i_{t}$, for $t=0,1,...,T-1$,
we have $\mathbf{E}_{i_{t}}\left[\left\langle G_{t},x^{*}-x_{t}\right\rangle _{\left(i_{t}\right)}\right]=\frac{1}{n}\left\langle G_{t},x^{*}-x_{t}\right\rangle $,
and $\mathbf{E}_{i_{t}}\left[\omega_{i_{t}^{c}}\left(x_{t}\right)\right]=\omega\left(x_{t}\right)-\mathbf{E}_{i_{t}}\left[\omega_{i_{t}}\left(x_{t}\right)\right]=\frac{n-1}{n}\omega\left(x_{t}\right)$.
Moreover, by the assumptions $\mathbf{E}_{\xi_{t}}\left[F\left(x_{t},\xi_{t}\right)\right]=f\left(x_{t}\right)$,
$\mathbf{E}_{\xi_{t}}\left[G\left(x_{t},\xi_{t}\right)\right]=g\left(x_{t}\right)$.
Together with the definition (\ref{eq:defi-delta}), we see $E\left[\delta_{t}\right]=0$.
In addition, from the Cauchy-Schwarz inequality, we have $\left(\|g_{t}-G_{t}\|_{\left(i_{t}\right)}+2M_{i_{t}}\right)^{2}\le2\left(\|g_{t}-G_{t}\|_{\left(i_{t}\right)}^{2}+4M_{i_{t}}^{2}\right)$,
and  the expectation $E_{\xi_{t}}\left[\|g_{t}-G_{t}\|_{\left(i_{t}\right)}^{2}\right]\le E_{\xi_{t}}\|G_{t}\|_{\left(i_{t}\right)}^{2}\le M_{i_{t}}^{2}$.
Furthermore, since $\xi_{t}$ is independent of $\gamma_{t-1}$ and
$l_{t-1}$, we have 
\begin{eqnarray*}
\mathbf{E}\left[\frac{\left(\|g_{t}-G_{t}\|_{\left(i_{t}\right)}+2M_{i}\right)^{2}}{\gamma_{t-1}^{\left(i_{t}\right)}\rho+l_{t-1}^{\left(i_{t}\right)}\lambda}\right] & \le & \mathbf{E}\left[\mathbf{E}_{\xi_{t}}\left(\frac{2\left(\|g_{t}-G_{t}\|_{\left(i_{t}\right)}^{2}+4M_{i_{t}}^{2}\right)}{\gamma_{t-1}^{\left(i_{t}\right)}\rho+l_{t-1}^{\left(i_{t}\right)}\lambda}\right)\right]\\
 & \le & \mathbf{E}\left[\left(\frac{10M_{i_{t}}^{2}}{\gamma_{t-1}^{\left(i_{t}\right)}\rho+l_{t-1}^{\left(i_{t}\right)}\lambda}\right)\right]\\
 & = & \sum_{i=1}^{n}\mathbf{E}\left[\frac{10M_{i}^{2}}{n\left(\gamma_{t-1}^{\left(i\right)}\rho+l_{t-1}^{\left(i\right)}\lambda\right)}\right].
\end{eqnarray*}
Using these results, we obtain

\begin{eqnarray*}
\sum_{t=1}^{T}\left(\alpha_{t-1}-\frac{n-1}{n}\alpha_{t}\right)\mathbf{E}\left[\phi\left(x_{t}\right)-\phi\left(x\right)\right] & \le & \alpha_{0}\frac{n-1}{n}\left[\phi\left(x_{0}\right)-\phi\left(x\right)\right]+\sum_{i=1}^{n}\mathbf{E}\left[\gamma_{T-1}^{\left(i\right)}\right]d_{i}\left(x\right)\\
 &  & +\sum_{i=1}^{n}\frac{5M_{i}^{2}}{n}\sum_{t=0}^{T-1}\mathbf{E}\left[\frac{\alpha_{t}^{2}}{\left(\gamma_{t-1}^{\left(i\right)}\rho+l_{t-1}^{\left(i\right)}\lambda\right)}\right].
\end{eqnarray*}

\end{proof}
In Theorem~\ref{thm-composite-conv} we presented some general convergence
properties of SBDA-u for both stochastic convex and strongly convex
functions. It should be noted that the right side of (\ref{main-converg-composite})
employs expectations since both $\left\{ \gamma_{t}\right\} $ and
$\left\{ l_{t}\right\} $ can be  random. In the sequel, we describe
more specialized convergence rates for both cases. Let us take $x=x^{*}$
and use the assumption (\ref{eq:assum-D}) throughout the analysis.

\subsubsection*{Convergence rate when $\omega(x)$ is a simple convex function}

Firstly, we consider a constant stepsize policy and assume that $\gamma_{t}^{\left(i\right)}$
depends on $i$ and $T$ where $T$ is the iteration number. More
specifically, let $\alpha_{t}\equiv1$, and $\gamma_{t}^{\left(i\right)}\equiv\beta_{i}$
for some $\beta_{i}>0$,$1\le i\le n$, $-1\le t\le T$. Then $\mathbf{E}\left[\frac{\alpha_{t}^{2}}{\gamma_{t-1}^{\left(i\right)}\rho}\right]=\frac{1}{\rho\beta_{i}}$,
for $1\le i\le n$, and hence 
\[
\sum_{t=1}^{T}\mathbf{E}\left[\phi\left(x_{t}\right)-\phi\left(x^{*}\right)\right]\le\left(n-1\right)\left[\phi\left(x_{0}\right)-\phi\left(x^{*}\right)\right]+n\sum_{i=1}^{n}\beta_{i}D_{i}+T\sum_{i=1}^{n}\frac{5M_{i}^{2}}{\rho\beta_{i}}.
\]
Let us choose $\beta_{i}=\sqrt{\frac{5TM_{i}^{2}}{n\rho D_{i}}}$
for $i=1,2,\ldots,p$, to optimize the above function. We obtain an
upper bound on the error term: 
\[
\sum_{t=1}^{T}\mathbf{E}\left[\phi\left(x_{t}\right)-\phi\left(x^{*}\right)\right]\le\left(n-1\right)\left[\phi\left(x_{0}\right)-\phi\left(x^{*}\right)\right]+2\sqrt{\frac{5Tn}{\rho}}\sum_{i=1}^{n}\sqrt{M_{i}^{2}D_{i}}.
\]
If we use the average point $\bar{x}=\sum_{t=1}^{T}x_{t}/T$ as the
output, we obtain the expected optimization error: 
\[
\mathbf{E}\left[\phi\left(\bar{x}\right)-\phi\left(x^{*}\right)\right]\le\frac{n-1}{T}\left[\phi\left(x_{0}\right)-\phi\left(x^{*}\right)\right]+\frac{2\sqrt{5n}\left[\sum_{i=1}^{n}\sqrt{M_{i}^{2}D_{i}}\right]}{\sqrt{\rho}\sqrt{T}}.
\]

In addition, we can also choose varying stepsizes without knowing
ahead the iteration number $T$. Differing from traditional stepsize
policies where $\gamma_{t}$ is usually associated with $t$, here
$\left\{ \gamma_{t}^{\left(i\right)}\right\} $ is a random sequence
dependent on both $t$ and $i_{t}$. In order to establish the convergence
rate with such a randomized $\gamma_{t}$, we first state a useful
technical result.
\begin{lem}
\label{lemma-recursive-sum}Let $p$ be a real number with $0<p<1$,
$\left\{ a_{s}\right\} $ and $\left\{ b_{t}\right\} $ be sequences
of nonnegative numbers satisfying the relation: 
\[
a_{t}=pb_{t}+\left(1-p\right)a_{t-1},\quad t=1,2,\ldots
\]
Then 
\[
\sum_{s=0}^{t}a_{s}\le\sum_{s=1}^{t}b_{s}+\frac{a_{0}}{p}.
\]

\end{lem}
We first let $\alpha_{t}\equiv1$, and define $\left\{ \gamma_{t}\right\} $
recursively as 
\[
\gamma_{t}^{\left(i\right)}=\begin{cases}
u_{i}\sqrt{t+1} & i=i_{t}\\
\gamma_{t-1}^{\left(i\right)} & i\neq i_{t}
\end{cases},
\]
for some $u_{i}>0$, $i=1,2,...,n$, $t=0,1,2,\ldots,T$. From this
definition, we obtain 
\[
\mathbf{E}\left[\frac{1}{\gamma_{t-1}^{\left(i\right)}}\right]=\frac{1}{n}\frac{1}{u_{i}\sqrt{t}}+\frac{n-1}{n}\mathbf{E}\left[\frac{1}{\gamma_{t-2}^{\left(i\right)}}\right].
\]
Observing the fact that $\sum_{\tau=1}^{t}\frac{1}{\sqrt{\tau}}\le\int_{1}^{t+1}\frac{1}{\sqrt{x}}dx=2\sqrt{t+1}$
and applying Lemma~\ref{lemma-recursive-sum} with $a_{t}=\mathbf{E}\left[\frac{1}{\gamma_{t-1}^{\left(i\right)}}\right]$
and $b_{t}=\frac{1}{u_{i}\sqrt{t}}$, we have 
\[
\sum_{\tau=0}^{t}\mathbf{E}\left[\frac{1}{\gamma_{\tau-1}^{\left(i\right)}}\right]\le\frac{1}{u_{i}}\sum_{\tau=1}^{t}\frac{1}{\sqrt{\tau}}+\frac{n}{\gamma_{-1}^{\left(i\right)}}\le\frac{2\sqrt{t+1}}{u_{i}}+\frac{n}{\gamma_{-1}^{\left(i\right)}}.
\]
Hence
\begin{equation}
\sum_{t=0}^{T-1}\mathbf{E}\left[\frac{1}{\gamma_{t-1}^{\left(i\right)}\rho}\right]\le\frac{1}{\rho}\left[\frac{2\sqrt{T}}{u_{i}}+\frac{n}{\gamma_{-1}^{\left(i\right)}}\right],\quad i=1,2,\ldots n.\label{ineq:auxi}
\end{equation}
With respect to (\ref{ineq:auxi}) and Theorem 1 , we obtain 
\[
\sum_{i=1}^{n}\mathbf{E}\left[\gamma_{T-1}^{\left(i\right)}\right]d_{i}\left(x^{*}\right)+\sum_{t=0}^{T-1}\sum_{i=1}^{n}\mathbf{E}\left[\frac{5\alpha_{t}^{2}M_{i}^{2}}{n\gamma_{t-1}^{\left(i\right)}\rho}\right]\le\sum_{i=1}^{n}u_{i}\sqrt{T}D_{i}+\sum_{i=1}^{n}\left\{ \frac{5M_{i}^{2}}{n\rho}\left[\frac{2\sqrt{T}}{u_{i}}+\frac{n}{\gamma_{-1}^{\left(i\right)}}\right]\right\} .
\]
Choosing $u_{i}=\sqrt{\frac{10M_{i}^{2}}{n\rho D_{i}}}$, we have

\[
\mathbf{E}\left[\phi\left(\bar{x}\right)-\phi\left(x^{*}\right)\right]\le\frac{n-1}{T}\left[\phi\left(x_{0}\right)-\phi\left(x^{*}\right)\right]+\sum_{i=1}^{n}\frac{5nM_{i}^{2}}{\rho\gamma_{-1}^{\left(i\right)}T}+\frac{2\sum_{i=1}^{n}\sqrt{10nM_{i}^{2}D_{i}}}{\sqrt{\rho}\sqrt{T}}.
\]

We summarize the results in the following corollary:
\begin{cor}
\label{cor:conv-rate}In algorithm \ref{alg:sbda-u}, let $T>0$,
$\bar{x}$ be the average point $\bar{x}=\sum_{t=1}^{T}x_{t}/T$,
and $\alpha_{t}\equiv1$.\end{cor}
\begin{enumerate}
\item If $\gamma_{t}^{\left(i\right)}=\sqrt{\frac{5TM_{i}^{2}}{n\rho D_{i}}}$,
for $t=0,1,2,...,T-1$, $i=1,2,...,n$, then 
\[
\mathbf{E}\left[\phi\left(\bar{x}\right)-\phi\left(x^{*}\right)\right]\le\frac{\left(n-1\right)\left[\phi\left(x_{0}\right)-\phi\left(x^{*}\right)\right]}{T}+\frac{2\sum_{i=1}^{n}\sqrt{5nM_{i}^{2}D_{i}}}{\sqrt{\rho}\sqrt{T}};
\]

\item If $\gamma_{t}^{\left(i\right)}=\begin{cases}
\sqrt{\frac{10M_{i}^{2}\left(t+1\right)}{n\rho D_{i}}} & \mbox{if }i=i_{t}\\
\gamma_{t-1}^{\left(i\right)} & \text{o.w.}
\end{cases}$, for $t=0,1,2,...,T-1$, and $\gamma_{-1}^{\left(i\right)}=\sqrt{\frac{10M_{i}^{2}}{n\rho D_{i}}}$,
$i=1,2,...,n$, then 
\[
\mathbf{E}\left[\phi\left(\bar{x}\right)-\phi\left(x^{*}\right)\right]\le\frac{n-1}{T}\left[\phi\left(x_{0}\right)-\phi\left(x^{*}\right)\right]+\sum_{i=1}^{n}\frac{5nM_{i}^{2}}{\rho\gamma_{-1}^{\left(i\right)}T}+\frac{2\sum_{i=1}^{n}\sqrt{10nM_{i}^{2}D_{i}}}{\sqrt{\rho}\sqrt{T}}.
\]

\end{enumerate}
Corollary~\ref{cor:conv-rate} provides both constant and adaptive
stepsizes and SBDA-u obtains a rate of convergence of \footnotesize
$O\left(1/\sqrt{T}\right)$\normalsize for both, which matches the
optimal rate for nonsmooth stochastic approximation {[}please see
(2.48) in \cite{Nemirovski:2009:RSA:1654243.1654247}{]}. In the context
of nonsmooth deterministic problem, it also matches the convergence
rate of the subgradient method. However, it is more interesting to
compare this with the convergence rate of BCD methods {[}please see,
for example, Corollary 2.2 part b) in \cite{Dang:arXiv1309.2249}{]}.
Ignoring the higher order terms, their convergence rate reads: \scriptsize$O\left(\frac{\sqrt{\sum_{i=1}^{n}M_{i}^{2}}}{\sqrt{T}}\sqrt{n\sum_{i=1}^{n}D_{i}}\right)$\normalsize.
Although the rate of $O\left(1/\sqrt{T}\right)$ is unimprovable,
it can be seen (using the Cauchy-Schwarz inequality) that 
\[
\sum_{i=1}^{n}\sqrt{M_{i}^{2}D_{i}}\le\sqrt{\sum_{i=1}^{n}M_{i}^{2}}\sqrt{\sum_{i=1}^{n}D_{i}},
\]
with the equality holding if and only if the ratio $M_{i}^{2}/D_{i}$
is equal to some positive constant, $1\le i\le n$. However, if this
ratio is very different in each coordinate block, SBDA-u is able to
obtain a much tighter bound. To see this point, consider the sequences
$\{M_{i}\}$ and $\{D_{i}\}$ such that $k$ items in $\left\{ M_{i}\right\} $
are $O(\tilde{M})$ for some integer $k$, $0<k\ll n$, while the
rest are $o\left(1/n\right)$ and $D_{i}$ is uniformly bounded by
$\tilde{D}$, $1\le i\le n$. Then the constant in SBDA-u is $O(\sqrt{n}k\tilde{M}\sqrt{\tilde{D}})$
while the one in SBMD is $O(n\sqrt{k}\tilde{M}\sqrt{\tilde{D}})$,
which is $\sqrt{n/k}$ times larger.

\subsubsection*{Convergence rate when $\omega\left(x\right)$ is strongly convex }

In this section, we investigate the convergence of SBDA-u when $\omega\left(x\right)$
is strongly convex with modulus $\lambda$, $\lambda>0$. More specifically,
we consider two averaging schemes and stepsize selections. In the
first approach, we apply a simple averaging scheme similar to \cite{Xiao:2010:DAM}.
By setting $\alpha_{t}\equiv1$,  all the past stochastic block subgradients
are weighted equally. In the second approach we apply a more aggressive
weighting scheme, which puts more weights on the later iterates. 

To prove the convergence of SBDA-u when $\omega\left(x\right)$ is
strongly convex, we introduce in the following lemma, a useful \textquotedblleft coupling\textquotedblright{}
property for Bernoulli random variables:
\begin{lem}
\label{lem:coupling} Let $r_{1},r_{2},r_{3}$ be i.i.d. samples from
$\text{Bernoulli}\left(p\right)$, $0<p<1$, $a,b>0$, and any $x$,
such that $0\le x\le a$, then 
\begin{equation}
\mathbf{E}\left[\frac{1}{r_{1}x+r_{2}\left(a-x\right)+b}\right]\le\mathbf{E}\left[\frac{1}{r_{3}a+b}\right].\label{ineq:coupling}
\end{equation}

\end{lem}
In the next corollary, we derive these specific convergence rates
for strongly convex problems. 
\begin{cor}
\label{cor:conv-rate-strongly-convex}In algorithm \ref{alg:sbda-u}:
if $\omega\left(x\right)$ is $\lambda$-strongly convex with modulus
$\lambda>0$, then \end{cor}
\begin{enumerate}
\item if $\alpha_{t}\equiv1$, $\gamma_{t}^{\left(i\right)}=\lambda/\rho$,
for $t=0,1,2,...,T-1$, and $\bar{x}=\sum_{t=1}^{T}x_{t}/T$, then
\begin{eqnarray*}
\mathbf{E}\left[\phi\left(\bar{x}\right)-\phi\left(x^{*}\right)\right] & \le & \frac{\left(n-1\right)\left[\phi\left(x_{0}\right)-\phi\left(x^{*}\right)\right]+n\lambda/\rho\sum_{i=1}^{n}d_{i}\left(x^{*}\right)}{T}\\
 &  & +\frac{5n\left(\sum_{i=1}^{n}M_{i}^{2}\right)\log\left(T+1\right)}{\lambda T}.
\end{eqnarray*}

\item if $\alpha_{t}=n+t$, for $t=0,1,2,\ldots$, and $\alpha_{-1}=0$,
$\gamma_{t}^{\left(i\right)}=\lambda\left(2n+T\right)/\rho$, for
$t=0,1,2,...,T-1$, %
then 
\begin{eqnarray*}
\mathbf{E}\left[\phi\left(\bar{x}\right)-\phi\left(x^{*}\right)\right] & \le & \frac{2n\left(n-1\right)\left[\phi\left(x_{0}\right)-\phi\left(x^{*}\right)\right]+2n\left(2n+T\right)\lambda/\rho\sum_{i}^{n}d_{i}\left(x^{*}\right)}{T\left(T+1\right)}\\
 &  & +\frac{10n\left(\sum_{i=1}^{n}M_{i}^{2}\right)}{\lambda\left(T+1\right)}\left[1+\frac{n+\left(n+1\right)\log T}{T}\right].
\end{eqnarray*}
\end{enumerate}
\begin{proof}
In part 1), let $\alpha_{t}\equiv1$, $\gamma_{t}^{\left(i\right)}\equiv\lambda/\rho$,
 it can be observed that $l_{t-1}^{\left(i\right)}\sim\text{Binomial}\left(t,\frac{1}{n},\frac{n-1}{n}\right)$
 $t\ge0$, we have
\begin{eqnarray*}
\mathbf{E}\left[\frac{1}{l_{t-1}^{\left(i\right)}\lambda+\gamma_{t-1}^{\left(i\right)}\rho}\right] & = & \sum_{i=0}^{t}\binom{t}{i}\left(\frac{1}{n}\right)^{i}\left(\frac{n-1}{n}\right)^{t-i}\frac{1}{\lambda\left(i+1\right)}\\
 & = & \frac{n}{\lambda\left(t+1\right)}\sum_{i=0}^{t}\binom{t+1}{i+1}\left(\frac{1}{n}\right)^{i+1}\left(\frac{n-1}{n}\right)^{t-i}\\
 & = & \frac{n}{\lambda\left(t+1\right)}\left[1-\left(\frac{n-1}{n}\right)^{t+1}\right]\\
 & \le & \frac{n}{\lambda\left(t+1\right)}.
\end{eqnarray*}
Observing the fact that $\sum_{\tau=0}^{t}\frac{1}{\tau+1}\le\int_{1}^{t+2}\frac{1}{x}dx\le\log\left(t+2\right),$
we obtain 
\begin{eqnarray*}
\mathbf{E}\left[\phi\left(\bar{x}\right)-\phi\left(x^{*}\right)\right] & \le & \frac{\left(n-1\right)\left[\phi\left(x_{0}\right)-\phi\left(x^{*}\right)\right]+\lambda n/\rho\sum_{i=1}^{n}d_{i}\left(x^{*}\right)}{T}+\frac{5n\left(\sum_{i=1}^{n}M_{i}^{2}\right)\log\left(T+1\right)}{\lambda T}.
\end{eqnarray*}

In part 2), let $\alpha_{t}=n+t$, for $t=0,1,2,\ldots$, and $\alpha_{-1}=0$,
$\gamma_{t}^{\left(i\right)}=\lambda\left(2n+T\right)/\rho$, then
for $t\ge0$, for any fixed $i$, let $r_{s}=\mathbf{1}_{i_{s}=i}$,
hence $r_{s}\sim\text{Bernoulli}\left(p\right)$. In addition, we
assume a sequence of ghost i.i.d. samples $\left\{ r'_{s}\right\} _{0\le s\le T}$.
For $t>0$, 
\begin{eqnarray}
\mathbf{E}\left[\frac{1}{l_{t}^{\left(i\right)}\lambda+\gamma_{t}^{\left(i\right)}\rho}\right] & = & \mathbf{E}\left[\frac{1}{\lambda\sum_{s=0}^{t}r_{s}\left(n+s\right)+\gamma_{t}^{\left(i\right)}\rho}\right]\nonumber \\
 & \le & \mathbf{E}\left[\frac{1}{\lambda\sum_{s=0}^{\lceil t/2\rceil-1}\left\{ r_{s}\left(n+s\right)+r_{t-s}\left(n+t-s\right)\right\} +\lambda\left(2n+T\right)}\right]\nonumber \\
 & \le & \mathbf{E}\left[\frac{1}{\lambda\left(2n+t\right)\left(\sum_{s=0}^{\lceil t/2\rceil-1}r'_{s}+1\right)}\right]\nonumber \\
 & \le & \frac{n}{\lambda\left(2n+t\right)\left(\max\left\{ \lceil t/2\rceil,1\right\} \right)}\label{ineq:bound-fractional-strong2}
\end{eqnarray}
where the second inequality follows from the independence of $\left\{ r_{s}\right\} $
and $\left\{ r'_{s}\right\} $ and the coupling property in Lemma~\ref{lem:coupling}.
It can be seen that the conclusion in (\ref{ineq:bound-fractional-strong2})
holds when $t=-1,0$ as well. Hence

\begin{eqnarray*}
\sum_{t=0}^{T-1}\mathbf{E}\left[\frac{\alpha_{t}^{2}}{\gamma_{t-1}^{\left(i\right)}\rho+l_{t-1}^{\left(i\right)}\lambda}\right] & \le & \sum_{t=0}^{T-1}\frac{\left(n+t\right)^{2}}{\lambda\left(2n+t-1\right)\left(\max\left\{ \lceil\left(t-1\right)/2\rceil,1\right\} \right)}\\
 & \le & \sum_{t=0}^{T-1}\frac{\left(n+t\right)}{\lambda\left(\max\left\{ \lceil\left(t-1\right)/2\rceil,1\right\} \right)}\\
 & = & \frac{2n+1}{\lambda}+\sum_{t=2}^{T-1}\frac{\left(n+t\right)}{\lambda\lceil\left(t-1\right)/2\rceil}\\
 & \le & \frac{2n+1}{\lambda}+\sum_{t=2}^{T-1}\frac{2\left(n+t\right)}{\lambda\left(t-1\right)}\\
 & = & \frac{2n+2T-1}{\lambda}+\sum_{t=2}^{T-1}\frac{2\left(n+1\right)}{\lambda\left(t-1\right)}\\
 & \le & \frac{2n+2T}{\lambda}+\frac{2\left(n+1\right)}{\lambda}\int_{1}^{T-1}\frac{1}{x}dx\\
 & \le & \frac{2}{\lambda}\left(n+T+\left(n+1\right)\log T\right).
\end{eqnarray*}
 Let $\bar{x}=\frac{\sum_{t=1}^{T}tx_{t}}{\sum_{t=1}^{T}t}$ be the
weighted average point, then 
\begin{eqnarray*}
\mathbf{E}\left[\phi\left(\bar{x}\right)-\phi\left(x^{*}\right)\right] & \le & \frac{2n\left(n-1\right)\left[\phi\left(x_{0}\right)-\phi\left(x^{*}\right)\right]+2n\left(2n+T\right)\lambda/\rho\sum_{i}^{n}D_{i}}{T\left(T+1\right)}\\
 &  & +\frac{10n\left(\sum_{i=1}^{n}M_{i}^{2}\right)}{\lambda\left(T+1\right)}\left[1+\frac{n+\left(n+1\right)\log T}{T}\right].
\end{eqnarray*}

\end{proof}
For nonsmooth and strongly convex objectives, we presented two options
to select $\left\{ \alpha_{t}\right\} $ and $\left\{ \gamma_{t}\right\} $.
These results seem to provide new insights on the dual averaging approach
as well. To see this, we consider SBDA-u when $n=1$. In the first
scheme, when $\alpha_{t}\equiv1$, the convergence rate of $O\left(\log T/T\right)$
is similar to the one in \cite{Xiao:2010:DAM}. In the second scheme
of Corollary \ref{cor:conv-rate-strongly-convex}, it shows that regularized
dual averaging methods can be easily improved to be optimal while
being equipped with a more aggressive averaging scheme. Our observation
suggests an alternative with rate $O\left(1/T\right)$ to the more
complicated accelerated scheme (\cite{ghadimi2012optimal,chen2012optimal}).
Such results seems new to the world of simple averaging methods, and
is on par with the recent discoveries for stochastic mirror descent
methods (\cite{nedic2014stochastic,Dang:arXiv1309.2249,hazan2014beyond,rakhlin2012making,lacoste2012simpler}).

\section{Nonuniformly randomized SBDA (SBDA-r)}

In this section we consider the general nonsmooth convex problem when
$\omega\left(x\right)=0$ or $\omega\left(x\right)$ is lumped into
$f\left(\cdot\right)$: 
\[
\min_{x\in X}\phi\left(x\right)=f\left(x\right),
\]
and show a variant of SBDA in which block coordinates are sampled
non-uniformly. More specifically, we assume the block coordinates
are i.i.d. sampled from a discrete distribution $\left\{ p_{i}\right\} _{1\le i\le n}$,
$0<p_{i}<1$, $1\le i\le n$. We describe in Algorithm~\ref{alg:sbda-r}
the nonuniformly randomized stochastic block dual averaging method
(SBDA-r). 

\begin{algorithm}
\KwIn{ convex function $f$, sequence of samples $\left\{ \xi_{t}\right\} $,
distribution $\left\{ p_{i}\right\} _{1\le i\le n}$;}

initialize $\alpha_{0}\in\mathbb{R}$, $\gamma_{-1}\in\mathbb{R}^{n}$,$\bar{G}=\mathbf{0}^{N}$
and $x_{0}=\arg\min_{x\in X}\sum_{i=1}^{n}\frac{\gamma_{-1}^{\left(i\right)}}{p_{i}}d_{i}\left(x^{\left(i\right)}\right)$\;

\For{ $t=0,1,\ldots,T-1$}{

sample a block $i_{t}\in\left\{ 1,2,\ldots,n\right\} $ with probability
$\text{Prob}\left(i_{t}=i\right)=p_{i}$\;

set $\gamma_{t}^{\left(i\right)}$, $i=1,2,...,n$\;

receive sample $\xi_{t}$ and update $\bar{G}$: $\bar{G}=\bar{G}+\frac{\alpha_{t}}{p_{i_{t}}}U_{i}G^{\left(i_{t}\right)}\left(x_{t},\xi_{t}\right)$\;

update $x_{t+1}^{\left(i_{t}\right)}=\arg\min_{x\in X_{i_{t}}}\left\{ \left\langle \bar{G}^{\left(i_{t}\right)},x\right\rangle +\frac{\gamma_{t}^{\left(i_{t}\right)}}{p_{i_{t}}}d_{i_{t}}\left(x\right)\right\} $\;

set $x_{t+1}^{\left(j\right)}=x_{t}^{\left(j\right)}$, $j\neq i_{t}$\;
} \KwOut{$\bar{x}=\left(\sum_{t=0}^{T}\alpha_{t}x_{t}\right)/\left(\sum_{t=0}^{T}\alpha_{t}\right)$;}

\protect\caption{Nonuniformly randomized stochastic block dual averaging (SBDA-r) method
\label{alg:sbda-r}}
\end{algorithm}

In the next theorem, we present the main convergence property of SBDA-r,
which expresses the bound of the expected optimization error as a
joint function of the sampling distribution $\{p_{i}\}$, and the
sequences $\{\alpha_{t}\}$, $\{\gamma_{t}\}$.
\begin{thm}
\label{thm-nonsmooth-conv} In algorithm \ref{alg:sbda-r}, let $\left\{ x_{t}\right\} $
be the generated solutions and $x^{*}$ be the optimal solution, $\left\{ \alpha_{t}\right\} $
be a sequence of positive numbers, $\left\{ \gamma_{t}\right\} $
be a sequence of vectors satisfying the assumption (\ref{eq:assumption-gamma})
. Let $\bar{x}=\frac{\sum_{t=0}^{T}\alpha_{t}x_{t}}{\sum_{t=0}^{T}\alpha_{t}}$
be the average point, then

\begin{equation}
\mathbf{E}\left[f\left(\bar{x}\right)-f\left(x\right)\right]\le\frac{1}{\sum_{t=0}^{T}\alpha_{t}}\left\{ \sum_{t=0}^{T}\sum_{i=1}^{n}\mathbf{E}\left[\frac{\alpha_{t}^{2}\|G_{t}\|_{\left(i\right),*}^{2}}{2\rho\gamma_{t-1}^{\left(i\right)}}\right]+\sum_{i=1}^{n}\frac{\mathbf{E}\left[\gamma_{T}^{\left(i\right)}\right]}{p_{i}}d_{i}\left(x\right)\right\} .\label{eq:convergence-convex-nonsmooth}
\end{equation}
\end{thm}
\begin{proof}
For the sake of simplicity, let us denote $A_{t}=\sum_{\tau=0}^{t}\alpha_{\tau}$,
for $t=0,1,2\ldots$. Based on the convexity of $f$, we have $f\left(\frac{\sum_{t=0}^{T}\alpha_{t}x_{t}}{A_{T}}\right)\le\frac{\sum_{t=0}^{T}\alpha_{t}f\left(x_{t}\right)}{A_{T}}$
and $f\left(x_{t}\right)\le f\left(x\right)+\left\langle g_{t},x_{t}-x\right\rangle $
for $x\in X$. Then

\begin{align}
A_{T}\left[f\left(\bar{x}\right)-f\left(x\right)\right] & \le\sum_{t=0}^{T}\alpha_{t}\left\langle g_{t},x_{t}-x\right\rangle \nonumber \\
 & \le\underset{\Delta_{1}}{\underbrace{\sum_{t=0}^{T}\frac{\alpha_{t}}{p_{i_{t}}}\left\langle U_{i_{t}}G_{t}^{\left(i_{t}\right)},x_{t}-x\right\rangle }}+\underset{\Delta_{2}}{\underbrace{\sum_{t=0}^{T}\alpha_{t}\left\langle g_{t}-\frac{1}{p_{i_{t}}}U_{i_{t}}G_{t}^{\left(i_{t}\right)},x_{t}-x\right\rangle }}.\label{eq:recursive-0}
\end{align}
It suffices to  provide precise bounds on the expectation of $\Delta_{1}$,
$\Delta_{2}$ separately. 

We define the auxiliary function 
\[
\Psi_{t}\left(x\right)=\begin{cases}
\sum_{s=0}^{t}\frac{\alpha_{s}}{p_{i_{s}}}\left\langle U_{i_{s}}G_{s}^{\left(i_{s}\right)},x\right\rangle +\sum_{i=1}^{n}\frac{\gamma_{t}^{\left(i\right)}}{p_{i}}d_{i}\left(x^{\left(i\right)}\right) & t\ge0\\
\sum_{i=1}^{n}\frac{\gamma_{t}^{\left(i\right)}}{p_{i}}d_{i}\left(x^{\left(i\right)}\right) & t=-1
\end{cases}.
\]
 Thus 
\begin{eqnarray}
\Psi_{t}\left(x_{t+1}\right) & = & \min_{x}\Psi_{t}\left(x\right)\nonumber \\
 & \ge & \min_{x}\left\{ \sum_{s=0}^{t}\frac{\alpha_{s}}{p_{i_{s}}}\left\langle U_{i_{t}}G_{s}^{\left(i_{t}\right)},x\right\rangle +\sum_{i=1}^{n}\frac{\gamma_{t-1}^{\left(i\right)}}{p_{i}}d_{i}\left(x^{\left(i\right)}\right)\right\} \nonumber \\
 & = & \min_{x}\left\{ \frac{\alpha_{t}}{p_{i_{t}}}\left\langle U_{i_{t}}G_{t}^{\left(i_{t}\right)},x\right\rangle +\Psi_{t-1}\left(x\right)\right\} \label{eq:phi_t}
\end{eqnarray}
The first inequality follows from the property (\ref{eq:assumption-gamma}).
Next, using  (\ref{eq:phi_t}) and Lemma~\ref{lem:prox-ineq-2},
we obtain 
\[
\frac{\alpha_{t}}{p_{i_{t}}}\left\langle U_{i_{t}}G_{t}^{\left(i_{t}\right)},x_{t}\right\rangle \le\Psi_{t}\left(x_{t+1}\right)-\Psi_{t-1}\left(x_{t}\right)+\frac{\alpha_{t}^{2}}{2\rho p_{i_{t}}\gamma_{t-1}^{\left(i_{t}\right)}}\|G_{t}\|_{\left(i_{t}\right),*}^{2}.
\]
Summing up the above inequality for $t=0,\ldots,T$, we have

\begin{align}
\sum_{t=0}^{T}\frac{\alpha_{t}}{p_{i_{t}}}\left\langle U_{i_{t}}G_{t}^{\left(i_{t}\right)},x_{t}\right\rangle  & \le\Psi_{T}\left(x_{T+1}\right)-\Psi_{-1}\left(x_{0}\right)+\sum_{t=0}^{T}\frac{\alpha_{t}^{2}}{2\rho p_{i_{t}}\gamma_{t-1}^{\left(i_{t}\right)}}\|G_{t}\|_{\left(i_{t}\right),*}^{2}.\label{eq:delta1-bound1}
\end{align}
Moreover, by the optimality of $x_{T+1}$ in solving $\min_{x}\Psi_{T}\left(x\right)$,
for all $x\in X$, we have 
\begin{equation}
\Psi_{T}\left(x_{T+1}\right)\le\sum_{t=0}^{T}\frac{\alpha_{t}}{p_{i_{t}}}\left\langle U_{i_{t}}G_{t}^{\left(i_{t}\right)},x\right\rangle +\sum_{i=1}^{n}\frac{\gamma_{T}^{\left(i\right)}}{p_{i}}d_{i}\left(x\right).\label{eq:delta1-bound2}
\end{equation}
Putting (\ref{eq:delta1-bound1}) and (\ref{eq:delta1-bound2}) together,
and using the fact that $\Psi_{-1}\left(x_{0}\right)\ge0$, we obtain:

\[
\Delta_{1}\le\sum_{i=1}^{n}\frac{\gamma_{T}^{\left(i\right)}}{p_{i}}d_{i}\left(x\right)+\sum_{t=0}^{T}\frac{\alpha_{t}^{2}}{2\rho p_{i_{t}}\gamma_{t-1}^{\left(i_{t}\right)}}\|G_{t}\|_{\left(i_{t}\right),*}^{2}.
\]
For each $t$, taking expectation w.r.t. $i_{t}$, we have

\begin{eqnarray*}
\mathbf{E}\left[\frac{\alpha_{t}^{2}}{2\rho p_{i_{t}}\gamma_{t-1}^{\left(i_{t}\right)}}\|G_{t}\|_{\left(i_{t}\right),*}^{2}\right] & = & \mathbf{E}\left[\mathbf{E}_{i_{t}}\left[\frac{\alpha_{t}^{2}}{2\rho p_{i_{t}}\gamma_{t-1}^{\left(i_{t}\right)}}\|G_{t}\|_{\left(i_{t}\right),*}^{2}\right]\right]\\
 & = & \sum_{i=1}^{n}\mathbf{E}\left[\frac{\alpha_{t}^{2}}{2\rho\gamma_{t-1}^{\left(i\right)}}\|G_{t}\|_{\left(i\right),*}^{2}\right].
\end{eqnarray*}
As a consequence, one has 
\begin{align}
\mathbf{E}\left[\Delta_{1}\right] & \le\sum_{i=1}^{n}\frac{\mathbf{E}\left[\gamma_{T}^{\left(i\right)}\right]}{p_{i}}d_{i}\left(x\right)+\sum_{t=0}^{T}\sum_{i=1}^{n}\mathbf{E}\left[\frac{\alpha_{t}^{2}\|G_{t}\|_{\left(i\right),*}^{2}}{2\rho\gamma_{t-1}^{\left(i\right)}}\right].\label{eq:expect-Delta1}
\end{align}
In addition, taking the expectation with respect to $i_{t}$, $\xi_{t}$
 and noting that $\mathbf{E}_{\xi_{t},i_{t}}\left[\frac{1}{p_{i_{t}}}U_{i_{t}}G_{t}\right]-g_{t}=\mathbf{E}_{\xi_{t}}\left[G_{t}\right]-g_{t}=0$,
we obtain 
\begin{equation}
\mathbf{E}\left[\Delta_{2}\right]=0.\label{eq:expect-Delta2}
\end{equation}
In view of (\ref{eq:expect-Delta1}) and (\ref{eq:expect-Delta2}),
we obtain the bound on the expected optimization error:

\[
\mathbf{E}\left[f\left(\bar{x}\right)-f\left(x\right)\right]\le\frac{1}{\sum_{t=0}^{T}\alpha_{t}}\left\{ \sum_{t=0}^{T}\sum_{i=1}^{n}\mathbf{E}\left[\frac{\alpha_{t}^{2}\|G_{t}\|_{\left(i\right),*}^{2}}{2\rho\gamma_{t-1}^{\left(i\right)}}\right]+\sum_{i=1}^{n}\frac{\mathbf{E}\left[\gamma_{T}^{\left(i\right)}\right]}{p_{i}}d_{i}\left(x\right)\right\} .
\]

\end{proof}

\subsubsection*{Block Coordinates Sampling and Analysis}

In view of Theorem 4, the obtained upper bound can be conceived as
a joint function of probability mass $\{p_{i}\}$, and the control
sequences $\{\alpha_{t}\}$, $\{\gamma_{t}\}$. Firstly, throughout
this section, let $x=x^{*}$ and  assume 
\begin{equation}
\alpha_{t}=1,\quad t=0,1,2,\ldots.\label{eq:alpha-value}
\end{equation}
 Naturally, we can choose the distribution and stepsizes by optimizing
the bound 
\begin{equation}
\min_{\{\gamma_{t}\},p}\mathcal{L}(\{\gamma_{t}\},p)=\sum_{t=0}^{T}\sum_{i=1}^{n}\mathbf{E}\left[\frac{M_{i}^{2}}{2\rho\gamma_{t-1}^{(i)}}\right]+\sum_{i}^{n}\frac{\mathbf{E}[\gamma_{T}^{(i)}]}{p_{i}}D_{i}.\label{p-gamma-0}
\end{equation}
This is a joint problem on  two groups of variables. Let us first
discuss how to choose $\left\{ \gamma_{t}\right\} $ for any fixed
$p_{i}$. Let us assume $p_{i}$ has the form: $p_{i}=\frac{M_{i}^{a}D_{i}^{b}}{C_{a,b}},\hspace{1em}i=1,2,\ldots,n,$
where $a,b\ge0$, and define $C_{a,b}=\sum_{i=1}^{n}M_{i}^{a}D_{i}^{b}$.
We derive two stepsizes rules, depending on whether the iteration
number $T$ is known or not. We assume $\gamma_{t}^{(i)}=\beta_{i}$,
for some constant $\beta_{i}$,  $i=1,2,\ldots n$, $t=1,2,...,T$.
The equivalent problem with $p$, $\beta$, has the form 
\begin{equation}
\min_{p,\beta}\mathcal{L}(p,\beta)=\sum_{i=1}^{n}\frac{(T+1)M_{i}^{2}}{2\rho\beta_{i}}+\sum_{i}^{n}\frac{\beta_{i}}{p_{i}}D_{i}.\label{p-gamma-1}
\end{equation}
By optimizing w.r.t. $\beta$, we obtain the optimal solutions 
\begin{equation}
\gamma_{t}^{(i)}=\beta_{i}=\sqrt{\frac{(1+T)p_{i}M_{i}^{2}}{2\rho D_{i}}}.\label{gamma-p}
\end{equation}
In addition, we can also select stepsizes without assuming the iteration
number $T$. Let us denote 
\begin{equation}
\gamma_{t}^{(i)}=\left\{ \begin{array}{ll}
\sqrt{t+1}u_{i} & \text{if }i=i_{t},\\
\gamma_{t-1}^{(i)} & \text{otherwise},
\end{array}\right.\label{eq:gamma_1_sbda-r}
\end{equation}
for some unspecified $u_{i}$, $1\le i\le n$. Applying Lemma~\ref{lemma-recursive-sum}
with $a_{t}=E\left[\frac{1}{\gamma_{t-1}^{(i)}}\right]$, $b_{t}=\frac{1}{u_{i}\sqrt{t}}$,
we have 
\[
\sum_{t=0}^{T}\mathbf{E}\left[\frac{1}{\gamma_{t-1}^{(i)}}\right]\le\sum_{t=1}^{T}\frac{1}{u_{i}\sqrt{t}}+\frac{1}{\gamma_{-1}^{(i)}p_{i}}\le2\frac{\sqrt{T+1}}{u_{i}}+\frac{1}{\gamma_{-1}^{(i)}p_{i}}.
\]
In view of the above analysis, we can relax the problem to the following:
\[
\min_{p,u}\sum_{i=1}^{n}\left[\frac{M_{i}^{2}\sqrt{T+1}}{\rho u_{i}}+\frac{u_{i}\sqrt{T+1}}{p_{i}}D_{i}+\frac{M_{i}^{2}}{2\rho\gamma_{-1}^{(i)}p_{i}}\right].
\]
Note that the third term above is $o\left(\sqrt{T}\right)$ and hence
can be ignored for the sake of simplicity.  Thus we have the approximate
problem 
\begin{equation}
\min_{p,u}\sum_{i=1}^{n}\left[\frac{M_{i}^{2}\sqrt{T+1}}{\rho u_{i}}+\frac{u_{i}\sqrt{T+1}}{p_{i}}D_{i}\right],\label{p-gamma-2}
\end{equation}
we apply the similar analysis and obtain $u_{i}=\sqrt{\frac{p_{i}M_{i}^{2}}{\rho D_{i}}}$
and hence the second stepsize rule

\begin{equation}
\gamma_{t}^{(i)}=\left\{ \begin{array}{ll}
\sqrt{\frac{(t+1)p_{i}M_{i}^{2}}{\rho D_{i}}} & \text{if }i=i_{t}\\
\gamma_{t-1}^{(i)} & \text{otherwise}
\end{array}\right.,\hspace{1em}t\ge0.\label{eq:gamma_2_sbda-r}
\end{equation}

We have established the relation between the optimized sampling probability
and stepsizes. Now we are ready to discuss specific choices of the
probability distribution. Firstly, the simplest way is to set 
\begin{equation}
p_{i}=\frac{1}{n},\quad i=1,2,\ldots.,n,\label{eq:uni-sampling}
\end{equation}
which implies that SBDA-r reduces to the uniform sampling method SBDA-u
with the obtained stepsizes entirely similar to the ones we derived
earlier. However, from the above analysis, it is possible to choose
the sampling distribution properly and obtain a further improved convergence
rate. Next we show how to obtain the optimal sampling and stepsize
policies from solving the joint problem (\ref{p-gamma-0}). We first
describe an important property in the following lemma.
\begin{lem}
\label{lemma:joint-problem}Let $\mathcal{S}^{n}$ be the $n$-dimensional
simplex. The optimal solution $x^{\ast}$, $y^{\ast}$ of the nonlinear
problem $\min_{x\in\mathbb{R}_{++}^{n},y\in\mathcal{S}^{n}}\sum_{i=1}^{n}\left[\frac{a_{i}}{x_{i}}+\frac{x_{i}}{b_{i}y_{i}}\right]$
where $a_{i},b_{i}>0$, $1\le i\le n$   is 
\[
y_{i}^{\ast}=(a_{i}/b_{i})^{\frac{1}{3}}W,\,\text{and }x_{i}^{\ast}=a_{i}^{\frac{2}{3}}b_{i}^{\frac{1}{3}}\sqrt{W},
\]
where $i=1,2,\ldots n$ and $W=\left(\sum_{i}^{n}(a_{i}/b_{i})^{\frac{1}{3}}\right)^{-1}$. 
\end{lem}
Applying Lemma \lemmaref{joint-problem} to the problem (\ref{p-gamma-1})
, we obtain the optimal sampling probability

\begin{equation}
p_{i}=M_{i}^{\frac{2}{3}}D_{i}^{\frac{1}{3}}/C,\quad i=1,2,\ldots n\label{eq:opt-sampling}
\end{equation}
where $C$ is the  normalizing constant. This is also the optimal
probability in problem (\ref{p-gamma-2}). In view of these results,
we obtain the specific convergence rates in the following corollary:
\begin{cor}
\label{cor:conv-rate-sbda-r}In algorithm \ref{alg:sbda-r}, let $\alpha_{t}=1$,
$t\ge0$. Denote $C=\left(\sum_{j=1}^{n}M_{j}^{2/3}D_{j}^{1/3}\right)$,
with block coordinates sampled from distribution (\ref{eq:opt-sampling}).
Then: \end{cor}
\begin{enumerate}
\item if $\gamma_{t}^{\left(i\right)}=\sqrt{\frac{\left(1+T\right)}{2\rho C}}M_{i}^{4/3}D_{i}^{-1/3}$,
$t\ge-1$, $i=1,2,\ldots,n$ , then 
\begin{equation}
\mathbf{E}\left[f\left(\bar{x}\right)-f\left(x^{*}\right)\right]\le\frac{\sqrt{2}}{\sqrt{\rho}}\frac{C^{3/2}}{\sqrt{T+1}}.\label{bound-nonuni1}
\end{equation}

\item if $\gamma_{-1}^{\left(i\right)}=\sqrt{\frac{1}{\rho C}}M_{i}^{4/3}D_{i}^{-1/3}$and
$\gamma_{t}^{(i)}=\begin{cases}
\sqrt{\frac{\left(t+1\right)}{\rho C}}M_{i}^{4/3}D_{i}^{-1/3} & \text{if }i=i_{t},\\
\gamma_{t-1}^{\left(i\right)} & \text{o.w.}
\end{cases}$, $t\ge0$, $i=1,2,\ldots,n$, then 
\begin{equation}
\mathbf{E}\left[f\left(\bar{x}\right)-f\left(x^{*}\right)\right]\le\frac{C^{3/2}}{\sqrt{\rho}}\left[\frac{2}{\sqrt{T+1}}+\frac{1}{2\left(T+1\right)}\right].\label{bound-nonuni2}
\end{equation}
\end{enumerate}
\begin{proof}
It remains to plug the value of $\left\{ \gamma_{t}\right\} $, $p$
back into $\mathcal{L}\left(,\right)$. 
\end{proof}
It is interesting to compare the convergence properties of SBDA-r
with that of SBDA-u and SBMD. SBDA with uniform sampling of block
coordinates only yields suboptimal dependence on the multiplicative
constants. Nevertheless, the rate can be further improved by employing
optimal nonuniform sampling. To develop further intuition, we relate
the two rates of convergence with the help of Hölder's inequality:
\[
\left[\sum_{i=1}^{n}\left(M_{i}^{2/3}D_{i}^{1/3}\right)\right]^{3/2}\le\left\{ \left[\sum_{i=1}^{n}\left(M_{i}^{2/3}D_{i}^{1/3}\right)^{3/2}\right]^{2/3}\cdot\left[\sum_{i=1}^{n}1^{3}\right]^{1/3}\right\} ^{3/2}=\sum_{i=1}^{n}\left(M_{i}\sqrt{D_{i}}\right)\cdot\sqrt{n}.
\]
The inequality is tight if and only if for some constant $c>0$ and
$i$, $1\le i\le n$: $M_{i}\sqrt{D_{i}}=c$. In addition, we compare
SBDA-r with a nonuniform version of SBMD\footnote{See Corollary 2.2, part a) of \cite{Dang:arXiv1309.2249}},
which obtains \scriptsize $O\left(\frac{\sqrt{\sum_{i=1}^{n}M_{i}^{2}}\cdot\sum_{i=1}^{n}\sqrt{D_{i}}}{\sqrt{T}}\right)$\normalsize,
assuming blocks are sampled based on the distribution $p_{i}\propto\sqrt{D_{i}}$.
Again, applying Hölder's inequality, we have
\[
\left[\sum_{i=1}^{n}\left(M_{i}^{2/3}D_{i}^{1/3}\right)\right]^{3/2}\le\left\{ \left[\sum_{i=1}^{n}\left(M_{i}^{2/3}\right)^{3}\right]^{1/3}\cdot\left[\sum_{i=1}^{n}\left(D_{i}^{1/3}\right)^{3/2}\right]^{2/3}\right\} ^{3/2}=\sqrt{\sum_{i=1}^{n}M_{i}^{2}}\cdot\sum_{i=1}^{n}\sqrt{D_{i}}.
\]

In conclusion, SBDA-r, equipped with an optimized block sampling scheme,
obtains the \textit{best} iteration complexity among all the block
subgradient methods.

\section{Experiments}

In this section, we examine the theoretical advantages of SBDA through
several preliminary experiments. For all the algorithms compared,
we estimate the parameters and tune the best stepsizes using separate
validation data. We first investigate the performance of SBDA on nonsmooth
deterministic problems by comparing its performance against other
nonsmooth algorithms. We compare with the following algorithms: SM1
and SM2 are subgradient mirror decent methods with stepsizes $\gamma_{1}\propto\frac{1}{\sqrt{t}}$
and $\gamma_{2}\propto\frac{1}{\|g\left(x\right)\|}$ respectively.
Finally, SGD is stochastic mirror descent and SDA a stochastic subgradient
dual averaging method. We study the problem of robust linear regression
($\ell_{1}$ regression) with the objective $\phi\left(x\right)=\frac{1}{m}\sum_{i=1}^{m}\left|b_{i}-a_{i}^{T}x\right|$.
The optimal solution $x^{*}$ and each $a_{i}$ are generated from
$\mathcal{N}\left(0,I_{n\times n}\right)$. In addition, we define
a scaling vector $s\in\mathbb{R}^{n}$ and $S$ a diagonal matrix
s.t. $S_{ii}=s_{i}$. We let $b=\left(AS\right)x^{*}+\sigma$, where
$A=\left[a_{2},a_{2},\ldots,a_{m}\right]^{T}\in\mathbb{R}^{m\times n}$,
and the noise $\sigma\sim\mathcal{N}\left(0,\rho I\right)$ . We set
$\rho=0.01$ and $m=n=5000$.

We plot the optimization objective with the number of passes of the
dataset in Figure~\ref{fig:l1_regression}, for four different choices
of $s$. In the first test case (leftmost subfigure), we let $s=\left[1,1,\ldots,1\right]^{T}$
so that columns of $A$ correspond to uniform scaling. We find that
SBDA-u and SBDA-r have slightly better performance than the other
algorithms while exhibiting very similar performance. In the next
three cases, $s$ is generated from the distribution $p\left(x;a\right)=a\left(1-x\right)^{a-1}$,
$0\le x\le1$, $a>0$. We set $a=1,5,30$ respectively. Employing
a large $a$ ensures that the bounds on the norms of block subgradients
follow the power law. We observe that stochastic methods outperform
the deterministic methods, and SBDA-based algorithms have comparable
and often better performance than SGD algorithms. In particular, SBDA-r
exhibits the best performance, which clearly shows the advantage of
SBDA with the nonuniform sampling scheme. 

\begin{figure}
\includegraphics[scale=0.2]{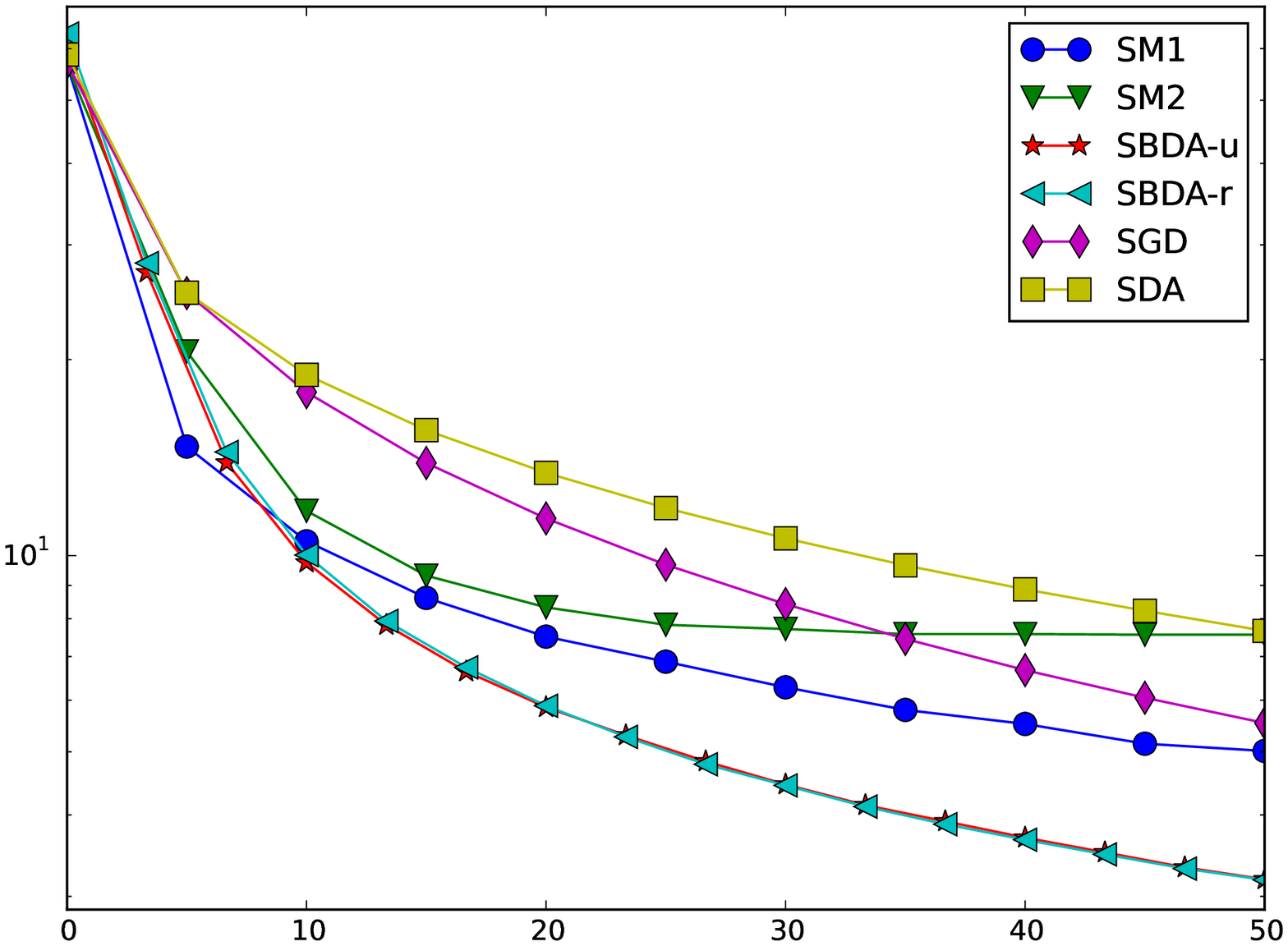}\includegraphics[scale=0.2]{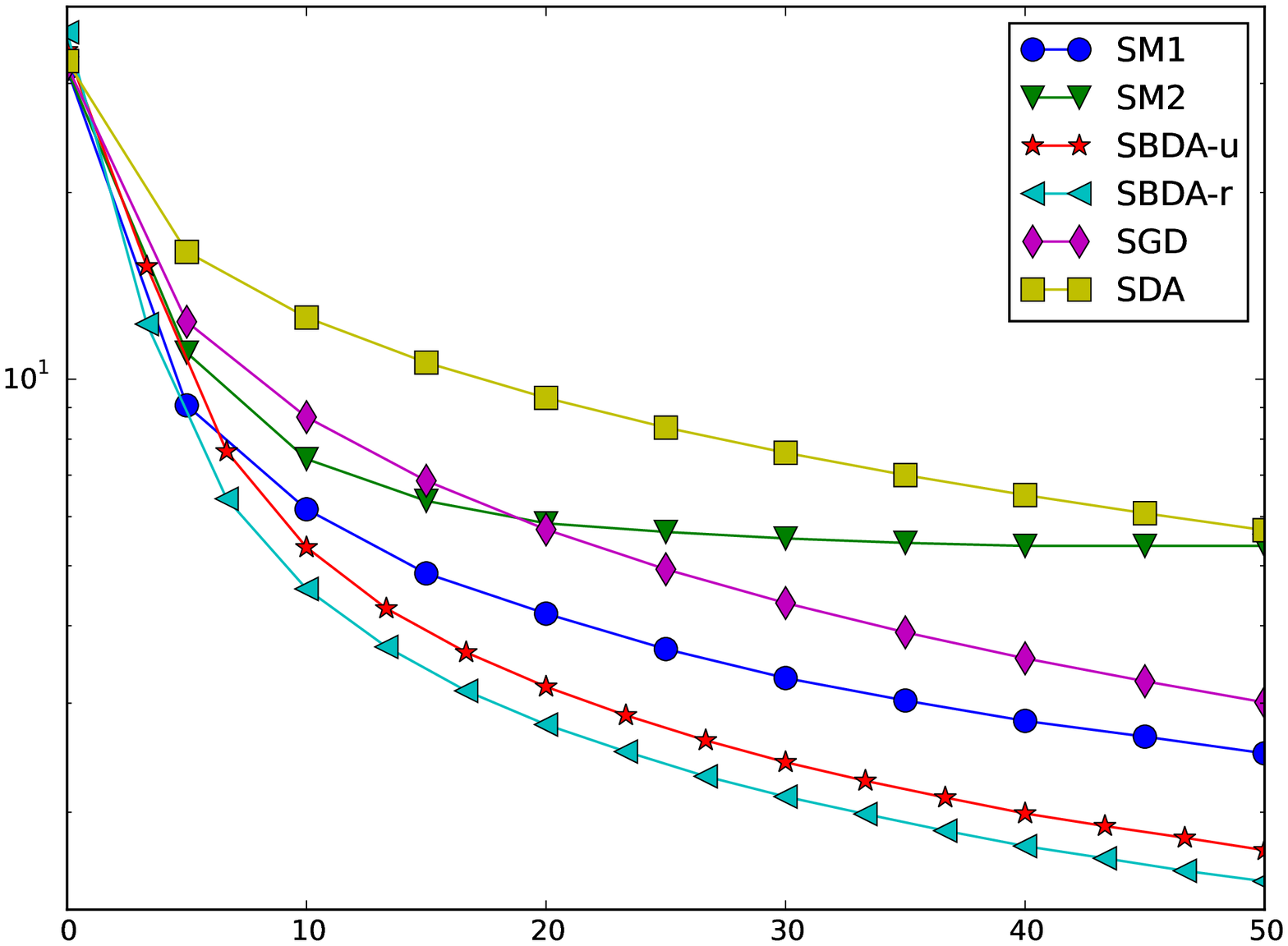}\includegraphics[scale=0.2]{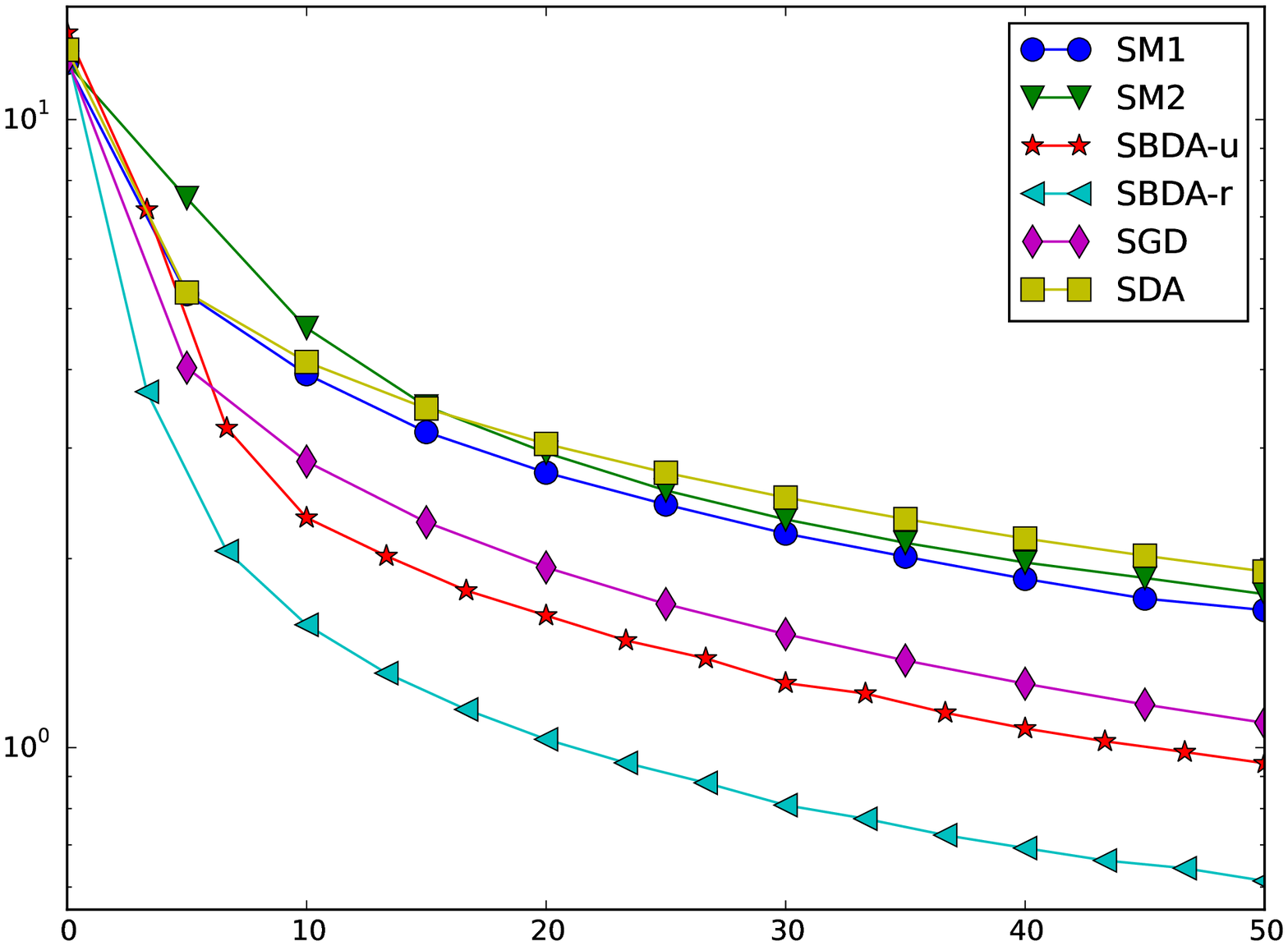}\includegraphics[scale=0.2]{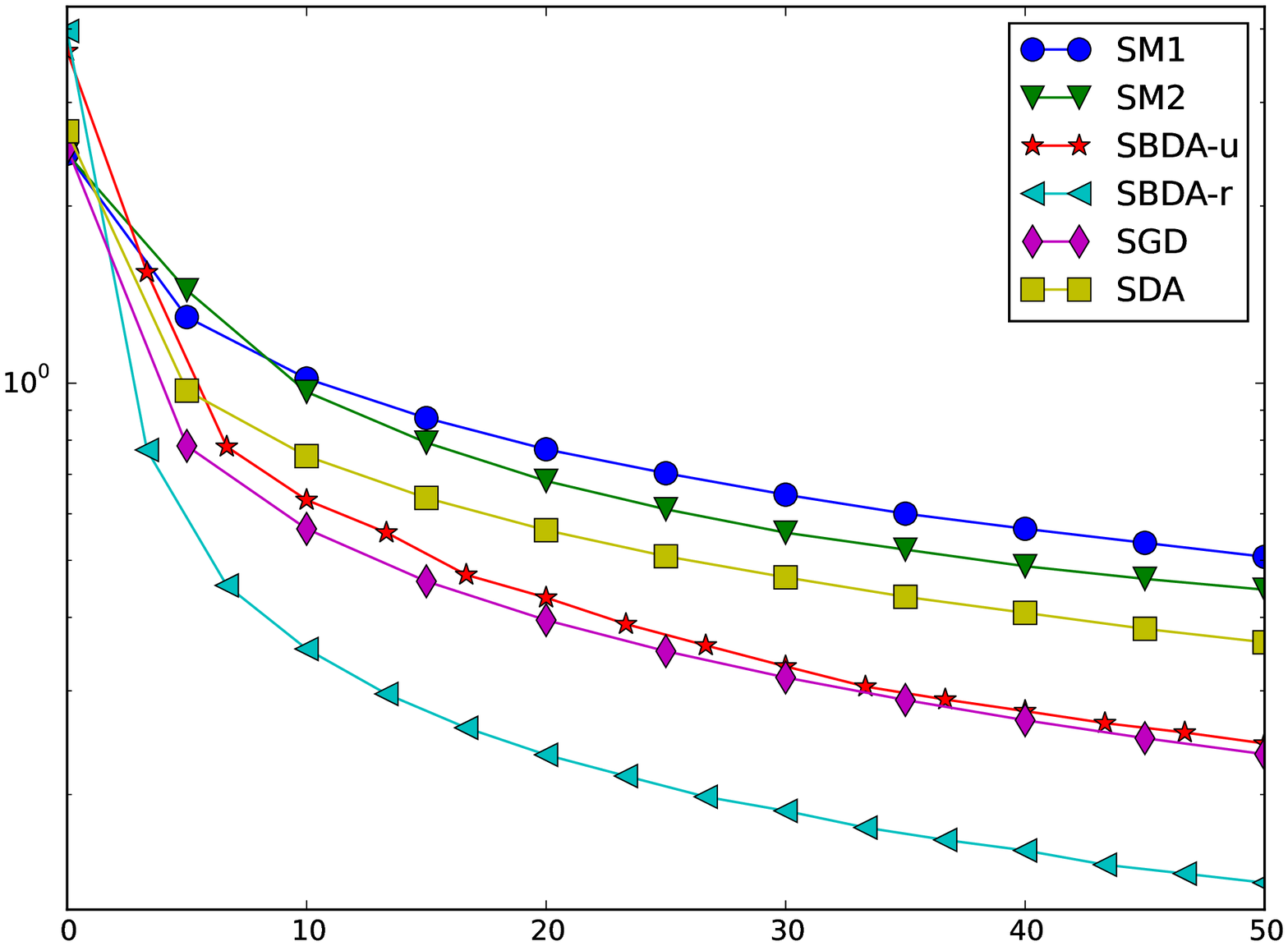}\caption{\label{fig:l1_regression} Tests on $\ell_{1}$ regression.}
\end{figure}

Next, we examine the performance of SBDA for online learning and stochastic
approximation. We conduct simulated experiments on the problem: $\phi\left(x\right)=\mathbf{E}_{a,b}\left[(b-\langle La,x\rangle)^{2}\right],$
where the aim is to fit linear regression under a linear transform
$L$. The transform matrix $L\in M^{n\times n}$ is generated as follows:
we first sample a matrix $\tilde{L}$ for which each entry $\tilde{L}_{i,j}\sim\mathcal{N}(0,1)$.
$L$ is obtained from $\tilde{L}$ with $90\%$ of the rows being
randomly rescaled by a factor $\rho$. To obtain the optimal solution
$x^{*}$, we first generate a random vector from the distribution
$\mathcal{N}(0,I_{n\times n})$ and then truncate each coordinate
in $[-1,1]$. Simulated samples are generated according to $b=\langle La,x^{*}\rangle+\varepsilon$
where $\varepsilon\in\mathcal{N}(0,0.01I_{n\times n})$. We let $n=200$,
and generate 3000 independent samples for training and 10000 independent
samples for testing. 

To compare the performances of these algorithms under various conditions,
we tune the parameter $\rho$ in $[1,0.1,0.05,0.01]$. As can be seen
from above, $\rho$ affects the estimation of block-wise parameters
$\left\{ M_{i}\right\} $. In Figure~\ref{synthetic_ls}, we show
the objective function for the average of 20 runs. The experimental
results show the advantages of SBDA over SBMD. When $\rho=1$, SBDA-u,
SBDA-r, and SBMD have the same theoretical convergence rate, and exhibit
similar performance. However, as $\rho$ decreases, the ``importance''
of $90\%$ of the blocks is diminishing and we find SBDA-u and SBDA-r
both outperform SBMD. Moreover, SBDA-r seems to perform the best,
suggesting the advantage of our proposed stepsize and sampling schemes
which are adaptive to the block structures. These observations lends
empirical support to our theoretical analysis. 
\begin{figure}
\includegraphics[scale=0.2]{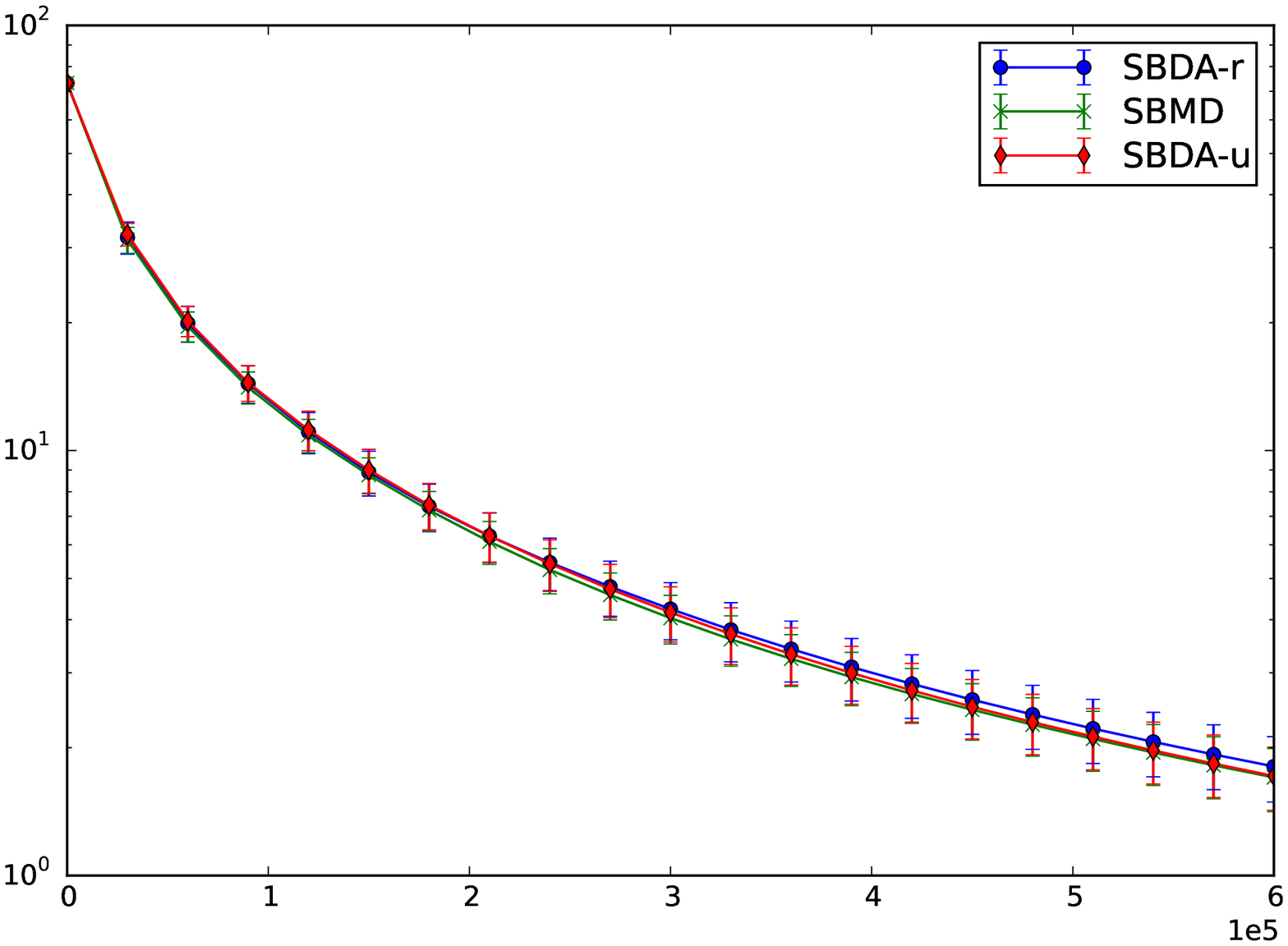}\includegraphics[scale=0.2]{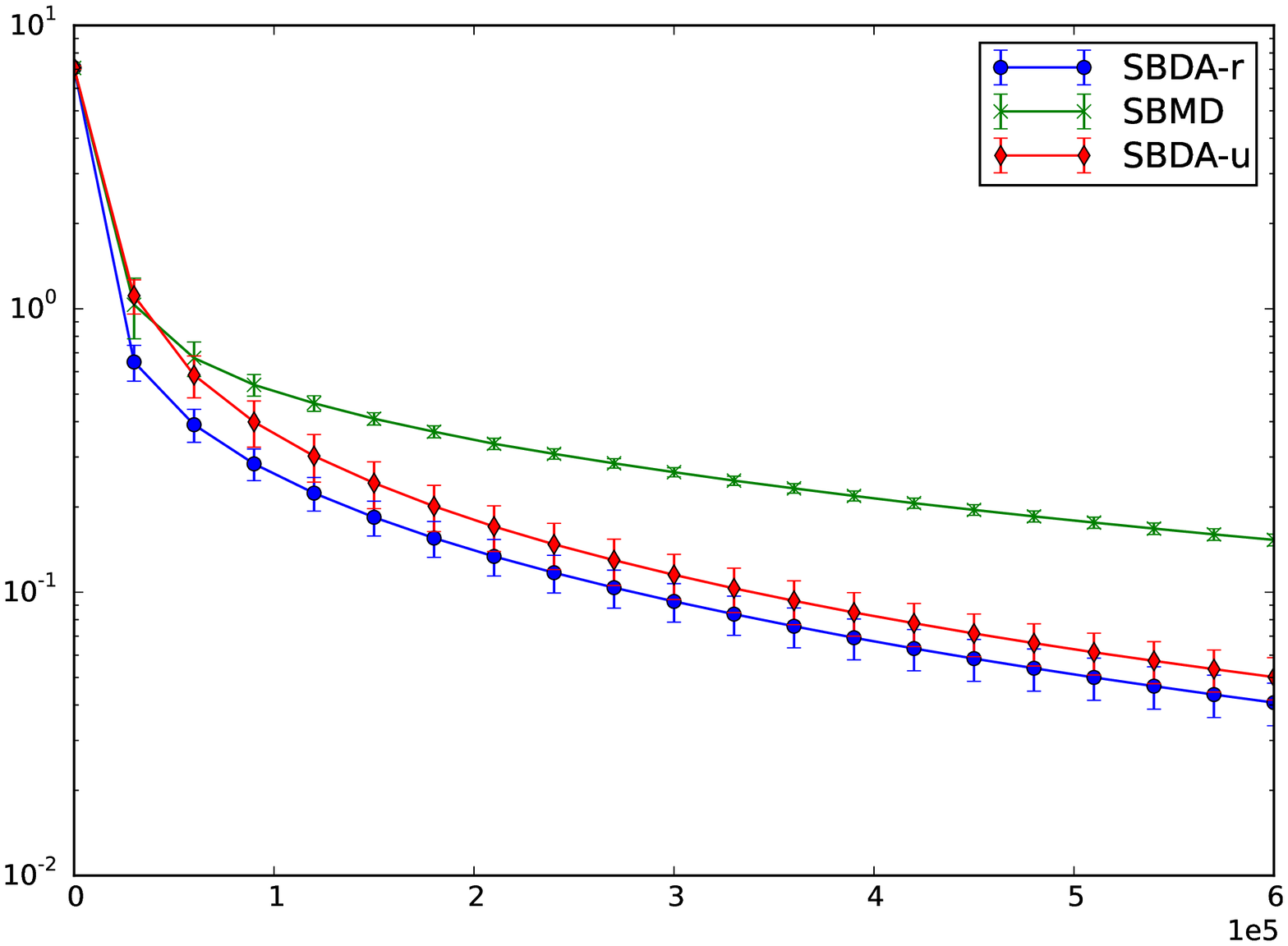}\includegraphics[scale=0.2]{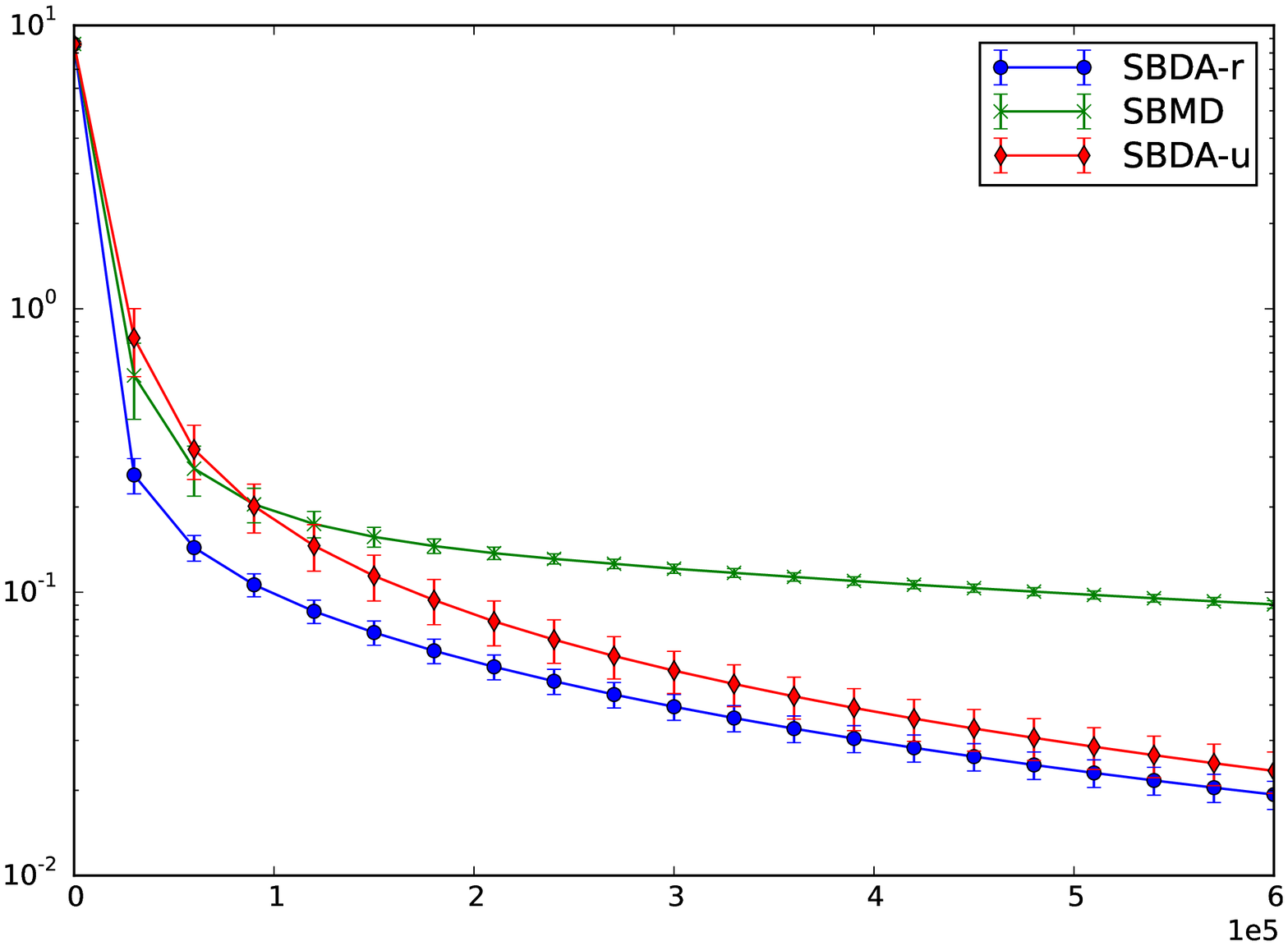}\includegraphics[scale=0.2]{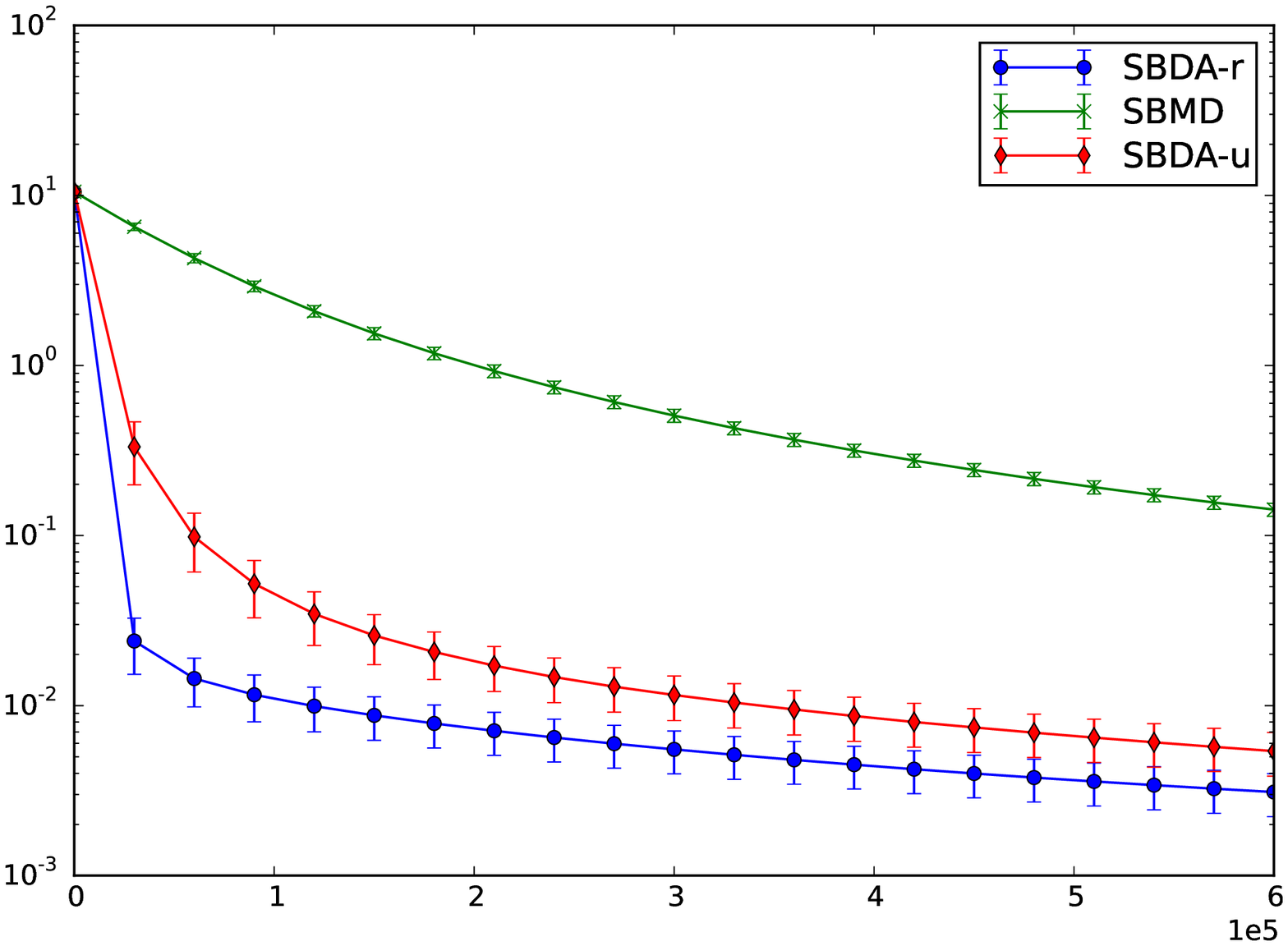}
\protect\caption{Tests on linear regression, Left to right: $\rho=1,0.1,0.05,0.01$.
\label{synthetic_ls}}
\end{figure}

Our next  experiment considers online $\ell_{1}$ regularized linear
regression (Lasso):\small 
\begin{equation}
\min_{w\in\mathbb{R}^{n}}\frac{1}{2}\mathbf{E}_{\left(y,x\right)}\left[\left(y-w^{T}x\right)^{2}\right]+\lambda\|w\|_{1}\label{eq:pob-regression}
\end{equation}
\normalsize While linear regression has been well studied in the
literature, recent work is interested in efficient regression algorithms
under different adversarial circumstances \cite{Cesa-Bianchi:2011:ELP:1953048.2078197,hazan2012linear,kukliansky2014attribute}.
Under the assumptions of limit budgets, the learner only partially
observes the features for each incoming instance, but is allowed to
choose the sampling distribution of the features. In addition, we
explicitly enforce the $\ell_{1}$ penalty, expecting to learn a sparse
solution that effectively reduces testing cost. To apply stochastic
methods, we estimate the stochastic coordinate gradient of the least
squares loss. For the sake of simplicity, we assume for each input
sample instance $\left(y,x\right)$, two features $\left(i_{t},j_{t}\right)$
are revealed. When we sample one coordinate $j_{t}$ from some distribution
$\left\{ p_{j}\right\} $, then $\frac{1}{p_{j_{t}}}w^{(j_{t})}x^{(j_{t})}$
is an unbiased estimator of $w^{T}x$. Hence the defined value $G^{\left(i_{t}\right)}=\frac{1}{p_{j_{t}}}x^{\left(i_{t}\right)}x^{\left(j_{t}\right)}w^{\left(j_{t}\right)}-yx^{\left(i_{t}\right)}$
is an unbiased estimator of the $i_{t}$-th coordinate gradient.

We adapt both SBMD and SBDA-u to these problems and conduct the experiments
on datasets \textsf{covtype} and \textsf{mnist} (digit ``3 vs 5'').
We also implement MD (composite mirror descent) and DA (regularized
dual averaging method). For all the methods, the training uses the
same total number of features. However, SBMD and SBDA-u obtain features
sampled using a uniform distribution; both MD and DA have ``unfair''
access to observe full feature vectors and therefore have the advantages
of lower variance. We plot in Figures~\ref{fig:lasso_1} and \ref{fig:lasso_2},
the optimization error and sparsity patterns with respect to the penalty
weights $\lambda$ on the two datasets. It can be seen that SBDA-u
has comparable and often better optimization accuracy than SBMD. In
addition, we also plot the sparsity patterns for different values
of $\lambda$. It can be seen that SBDA-u is very effective in enhancing
sparsity, more efficient than SBMD, MD, and comparable to DA which
doesn't have such budget constraints. 

\begin{figure}
\subfloat[Test on \textsf{covtype} dataset.\label{fig:lasso_1}]{\includegraphics[scale=0.2]{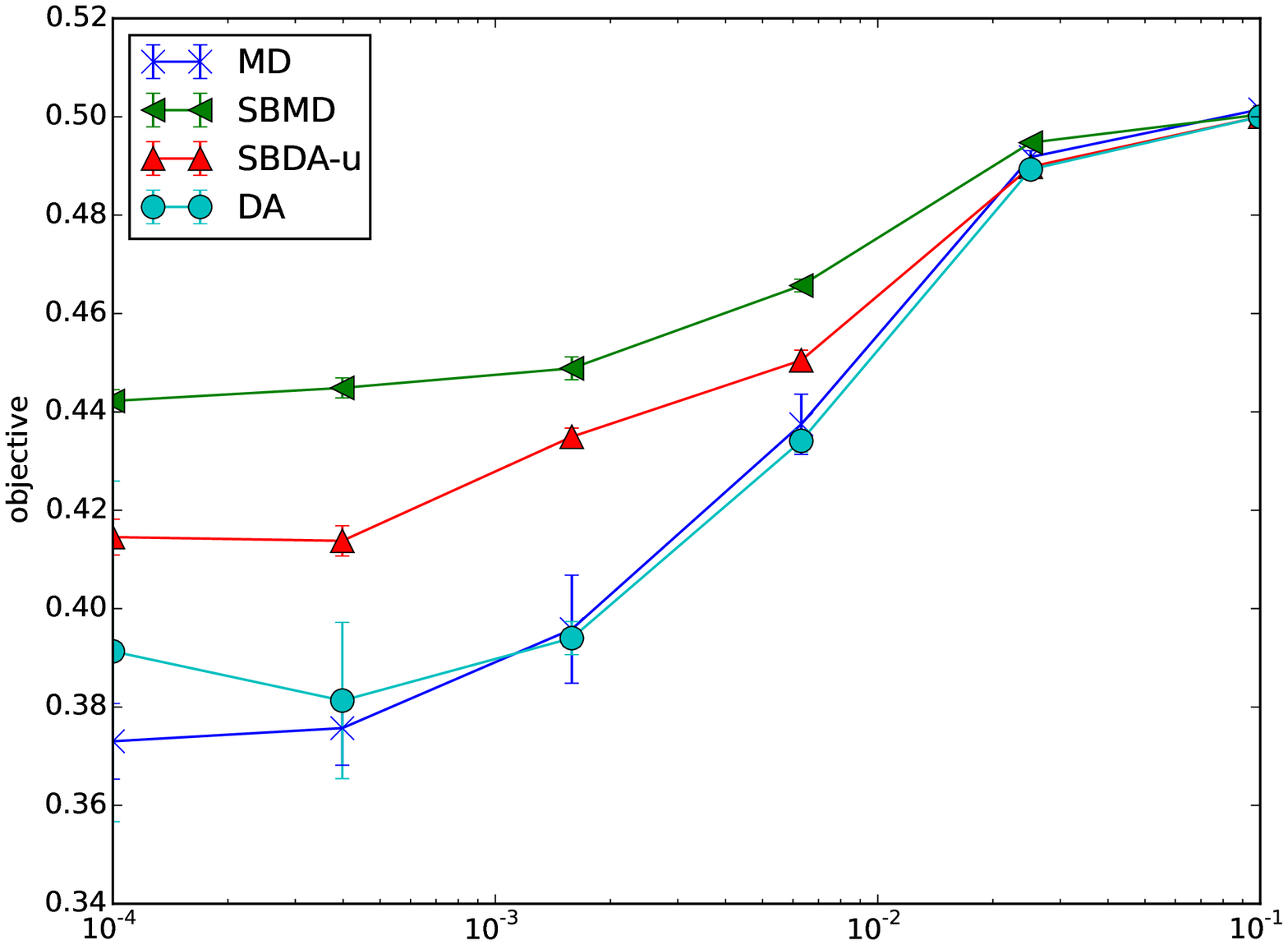}\includegraphics[scale=0.2]{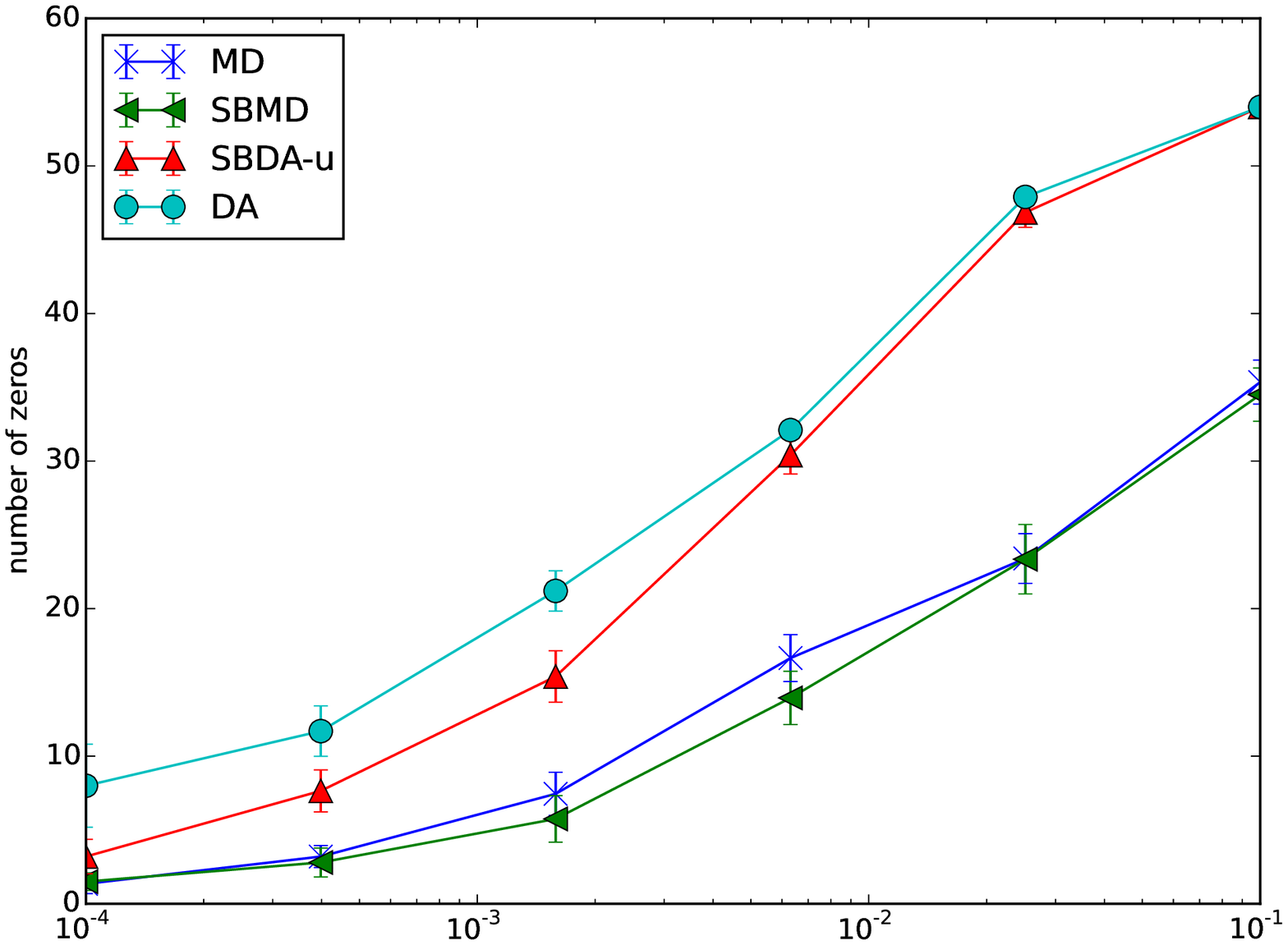}}\subfloat[Test on \textsf{mnist} dataset.\label{fig:lasso_2}]{\includegraphics[scale=0.2]{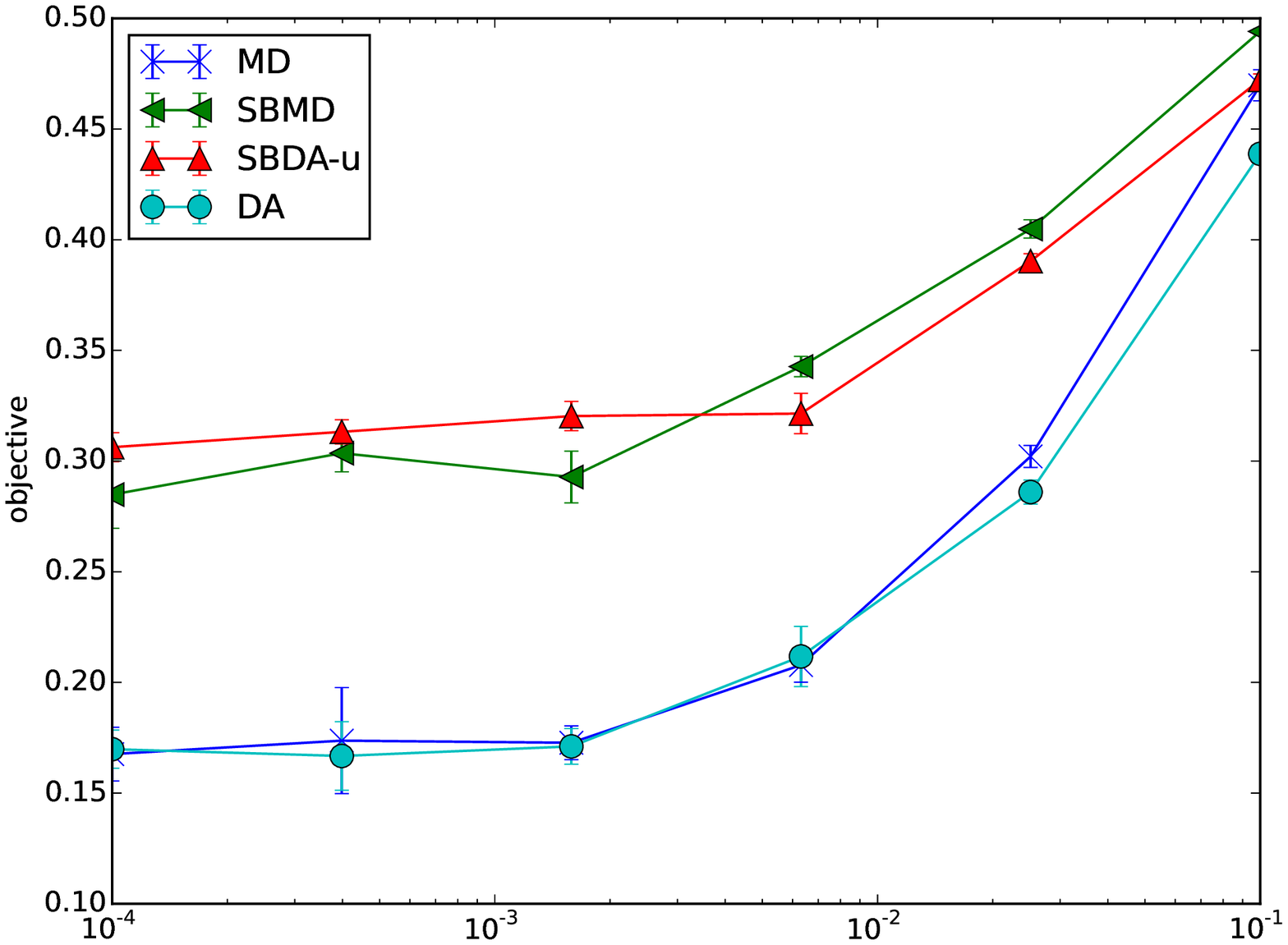}\includegraphics[scale=0.2]{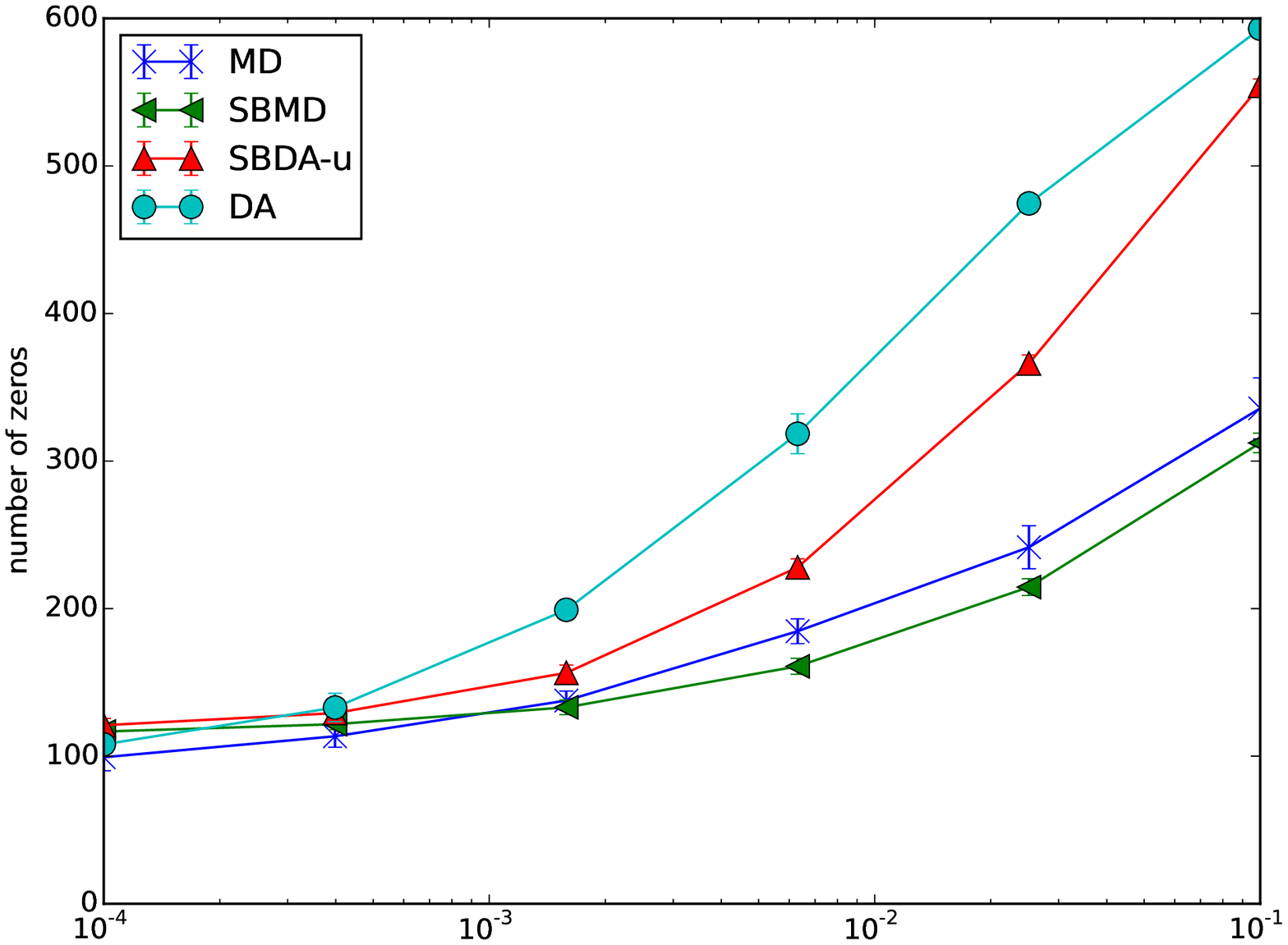}}

\caption{Tests on online lasso with limited budgets}
\end{figure}

\section{Discussion}

In this paper we introduced SBDA, a new family of block subgradient
methods for nonsmooth and stochastic optimization, based on a novel
extension of dual averaging methods. We specialized SBDA-u for regularized
problems with nonsmooth or strongly convex regularizers, and SBDA-r
for general nonsmooth problems. We proposed novel randomized stepsizes
and optimal sampling schemes which are truly block adaptive, and thereby
obtain a set of sharper bounds. Experiments demonstrate the advantage
of SBDA methods compared with subgradient methods on nonsmooth deterministic
and stochastic optimization. In the future, we will extend SBDA to
an important class of regularized learning problems consisting of
the finite sum of differentiable losses. On such problems, recent
work \cite{shalev2012stochastic,shalev2013accelerated} shows efficient
BCD convergence at linear rate. The works in \cite{zhao2014accelerated,wang2014randomized}
propose randomized BCD methods that sample both primal and dual variables.
However both methods apply conservative stepsizes which take the maximum
of the block Lipschitz constant. It would be interesting to see whether
our techniques of block-wise stepsizes and nonuniform sampling can
be applied in this setting as well to obtain improved performance.

\section{Appendix}

\subsection*{Proof of Lemma \ref{lem:prox-ineq-1}}
\begin{proof}
The first part comes from \cite{tseng2008accelerated}. Let $g\left(z\right)$
denote any subgradient of $f$ at $z$. Since $f\left(x\right)$ is
strongly convex, we have $f\left(x\right)\ge f\left(z\right)+\langle U_{i}g^{\left(i\right)}\left(z\right),x-z\rangle+\frac{\lambda}{2}\|x-z\|_{\left(i\right)}^{2}$.
By the definition of $z$ and optimality condition, we have $g^{\left(i\right)}\left(z\right)=-\nabla_{i}d\left(z\right)$.
Thus 
\[
f\left(x\right)+\left\langle \nabla_{i}d\left(z\right),x-z\right\rangle \ge f\left(z\right)+\frac{\lambda}{2}\|x-z\|_{\left(i\right)}^{2}.
\]
It remains to apply the definition $x=z+U_{i}y$ and $\mathcal{V}\left(z,x\right)=d\left(x\right)-d\left(z\right)-\left\langle \nabla d\left(z\right),x-z\right\rangle $.
\end{proof}

\subsection*{Proof of Lemma \ref{lem:prox-ineq-2}}
\begin{proof}
Let $h\left(y\right)=\max_{x\in X}\left\{ \left\langle y,x\right\rangle -\Psi\left(x\right)\right\} $,
since $\Psi\left(\cdot\right)$ is strongly convex and separable,
$h\left(\cdot\right)$ is convex and differentiable and its $i$-th
block gradient $\nabla_{i}h\left(\cdot\right)$ is $\frac{1}{\rho_{i}}$-smooth
. Moreover, we have $\nabla h\left(0\right)=x_{0}$ by the definition
of $x_{0}$. Thus 
\[
h\left(-U_{i}g^{\left(i\right)}\right)\le h\left(0\right)+\left\langle x_{0},-U_{i}g^{\left(i\right)}\right\rangle +\frac{1}{2\rho_{i}}\|g\|_{\left(i\right),*}^{2}.
\]
It remains to plug in the definition of $h\left(\cdot\right)$, $z$,
$x_{0}$.
\end{proof}

\subsection*{Proof of Lemma \ref{lem:function-bound}}
\begin{conjecture*}
By convexity of $f\left(\cdot\right)$, we have $f(z)\le f(x)+\langle g(z),z-x\rangle$.
In addition, 
\begin{eqnarray*}
\langle g(z),z-x\rangle & = & \langle g(x),z-x\rangle+\langle g(z)-g(x),z-x\rangle\\
 & = & \langle g^{(i)}(x),y\rangle_{(i)}+\langle g^{(i)}(z)-g^{(i)}(x),y\rangle_{(i)}\\
 & \le & \langle g^{(i)}(x),y\rangle_{(i)}+\|g^{(i)}(z)-g^{(i)}(x)\|_{(i),*}\cdot\|y\|_{(i)}.
\end{eqnarray*}
The second equation follows from the relation between $x,y,z$ and
the last one from the Cauchy-Schwarz inequality. Finally the conclusion
directly follows from (\ref{ass-grad-bound}).
\end{conjecture*}

\subsection*{Proof of Lemma \ref{lemma-recursive-sum}}
\begin{proof}
Let $A_{t}=\sum_{s=0}^{t}a_{t}$, $B_{t}=\sum_{s=1}^{t}b_{t}$. It
is equivalent to show $A_{t}\le B_{t}+\frac{A_{0}}{p}$. Then

\begin{align*}
A_{t} & =pB_{t}+A_{0}+\left(1-p\right)A_{t-1}\\
 & =\left[p+\left(1-p\right)\right]\left[pB_{t-1}+A_{0}\right]+\left(1-p\right)^{2}A_{t-2}\\
 & =\left[p+\left(1-p\right)+\left(1-p\right)^{2}\right]\left[pB_{t-1}+A_{0}\right]+\left(1-p\right)^{3}A_{t-3}\\
 & =...\\
\le & \left[pB_{t}+A_{0}\right]\left[\sum_{s=0}^{t}\left(1-p\right)^{s}\right].
\end{align*}
The last inequality follows from the assumption that $B_{t}\ge B_{s}$
where $0\le s\le t$ and $A_{0}=a_{0}$. It remains to apply the inequality
$\sum_{s=0}^{t}\left(1-p\right)^{s}\le\sum_{s=0}^{\infty}\left(1-p\right)^{s}=\frac{1}{p}$.
\end{proof}

\subsection*{Proof of Lemma \ref{lem:coupling}}
\begin{proof}
If $r_{1},r_{2},r_{3}\sim\text{Bernoulli}\left(p\right)$, $c>0$,
$0<p<1$,

\begin{eqnarray*}
\mathbf{E}\left[\frac{1}{r_{1}x+r_{2}\left(a-x\right)+b}\right] & = & \frac{\left(1-p\right)^{2}}{b}+\frac{p\left(1-p\right)}{a-x+b}+\frac{p\left(1-p\right)}{x+b}+\frac{p^{2}}{a+b}\\
 & \le & \frac{\left(1-p\right)^{2}}{b}+\frac{p\left(1-p\right)}{a+b}+\frac{p\left(1-p\right)}{b}+\frac{p^{2}}{a+b}\\
 & = & \frac{1-p}{b}+\frac{p}{a+b}\\
 & = & \mathbf{E}\left[\frac{1}{r_{3}a+b}\right].
\end{eqnarray*}
To see the first inequality, let $f\left(x\right)=\frac{A}{x+c}+\frac{B}{a-x+c}$,
where $A,B>0$, it can be seen that $f\left(\cdot\right)$ is convex
in $\left[0,a\right]$, then $\max_{x\in\left[0,a\right]}f\left(x\right)=\max\left\{ f\left(0\right),f\left(a\right)\right\} $.
\end{proof}

\subsection*{Proof of Lemma \ref{lemma:joint-problem}}
\begin{proof}
Let $x^{*}$, $y^{*}$ be the optimal solution of $\min_{x,y}\mathcal{L}\left(x,y,a,b\right)$.
We consider two subproblems. Firstly, $x^{*}=\arg\min_{x}\mathcal{L}\left(x,y^{*},a,b\right)$.
Since $\frac{a_{i}}{x_{i}}+\frac{x_{i}}{b_{i}y_{i}^{*}}\ge2\sqrt{\frac{a_{i}}{b_{i}y_{i}^{*}}}$,
at optimality 
\begin{equation}
\frac{a_{i}}{x_{i}^{*}}=\frac{x_{i}^{*}}{b_{i}y_{i}^{*}}.\label{eq:partial-optimality1}
\end{equation}
On the other hand, $y^{*}$ is the minimizer of the problem $\min_{y}\mathcal{L}\left(x^{*},y,a,b\right)$.
Applying the Cauchy-Schwarz inequality to $\mathcal{L}\left(x^{*},y,a,b\right)$,
we obtain 
\[
\sum_{i=1}^{n}\frac{x_{i}^{*}}{b_{i}y_{i}}\cdot1=\sum_{i=1}^{n}\frac{x_{i}^{*}}{b_{i}y_{i}}\sum_{1}^{n}y_{i}\ge\sum_{i}^{n}\sqrt{\frac{x_{i}^{*}}{b_{i}y_{i}}}\sqrt{y_{i}}=\sum_{i}^{n}\sqrt{x_{i}^{*}}.
\]
At optimality, the equality holds for some scalar $C>0$, 
\begin{equation}
\frac{x_{i}^{*}}{b_{i}y_{i}^{*}}=Cy_{i}^{*},\quad i=1,2,\ldots,n.\label{eq:partial-optimality2}
\end{equation}
It remains to solve the equations (\ref{eq:partial-optimality1})
and (\ref{eq:partial-optimality2}) with the simplex constraint on
$y$.
\end{proof}
\bibliographystyle{abbrv}
\bibliography{reference}

\end{document}